\documentclass{amsart}

\usepackage{amsmath,amssymb,amsthm,bm,mathtools}
\usepackage[inline]{enumitem}
\usepackage{url}
\usepackage{xcolor}
\usepackage[T1]{fontenc}
\usepackage{lmodern,microtype}
\usepackage[margin=30mm]{geometry}
\usepackage[hidelinks]{hyperref}
\newcommand{\ain}{\alpha_{\mathrm{in}}}
\newcommand{\abd}{\alpha_{\mathrm{bd}}}
\newcommand{\ass}{\alpha_{\mathrm{ss}}}
\newcommand{\agh}{\alpha_{\mathrm{gh}}}

\theoremstyle{plain}
\newtheorem{theorem}{Theorem}[section]
\newtheorem{corollary}[theorem]{Corollary}
\newtheorem{lemma}[theorem]{Lemma}
\newtheorem{proposition}[theorem]{Proposition}
\theoremstyle{definition}
\newtheorem{definition}[theorem]{Definition}
\newtheorem{remark}[theorem]{Remark}
\newtheorem{example}[theorem]{Example}

\numberwithin{equation}{section}

\newcommand{\R}{\mathbb{R}}

\newcommand{\dist}{\mathrm{dist}}
\DeclareMathOperator*{\osc}{osc}
\DeclareMathOperator*{\spt}{supp}
\DeclareMathOperator{\divergence}{div}
\newcommand{\capacity}{\mathrm{cap}}
\newcommand{\cl}{\overline}
\newcommand{\loc}{\mathrm{loc}}

\title[Parabolic equations with interior Morrey data]
{Global H\"{o}lder solvability for parabolic equations with interior Morrey data}
\author{Takanobu Hara}
\date{October 27, 2026}

\begin{document}
\begin{abstract}
We prove global H\"{o}lder solvability for uniformly parabolic divergence-form equations
with locally finite measure data under an interior Morrey growth condition and a capacity density condition on the spatial domain. 
The parabolic distance to the initial and lateral boundary is assumed bounded on the source support. 
For nonnegative data, we establish necessity results that demonstrate the sharpness of the growth and domain assumptions. 
We also derive estimates for divergence forcing, prove solvability for equations with lower-order terms, and obtain convergence rates in periodic homogenization.
\end{abstract}
\maketitle

\section{Introduction}\label{sec:introduction}
\enlargethispage{4pt}
Let $D\subset\R^n$, $n\ge2$, be a domain, and let $0<T\le\infty$.
We consider the Cauchy--Dirichlet problem
\begin{equation}\label{eqn:intro-equation}
\begin{cases}
\mathcal Hu:=\partial_tu-\divergence(A\nabla u)=\nu&\text{in }D\times(0,T),\\
u=0&\text{on }\partial D\times(0,T),\\
u(\cdot,0)=0&\text{in }D.
\end{cases}
\end{equation}
Here $A=A(X)$ is a measurable real $n\times n$ matrix, not necessarily symmetric, satisfying
\begin{equation}\label{eqn:A}
 A(X)\zeta\cdot\zeta\ge |\zeta|^2,
 \qquad |A(X)\zeta|\le L|\zeta|
\end{equation}
for almost every $X=(x,t)\in D\times(0,T)$ and every $\zeta\in\R^n$, where $L\ge1$.
The datum $\nu$ is a signed Radon measure that is locally finite in $D\times(0,T)$.
Its total variation is not assumed to be finite up to the parabolic boundary.

{
The De Giorgi--Nash--Moser theory gives interior H\"older regularity for homogeneous divergence-form equations with bounded measurable coefficients; see \cite{LSU,Lieberman1996,Moser1964}.
For inhomogeneous equations, Gaussian bounds for fundamental solutions \cite{Aronson1967,Aronson1968} and homogeneous oscillation estimates motivate a Morrey growth condition on the forcing measure.
}
Suppose that, for some $0<\beta\le1$ and $M<\infty$,
\begin{equation}\label{eqn:morrey}
 |\nu|(Q(X,r))\le Mr^{n+\beta}
\end{equation}
for every $X\in D\times(0,T)$ and $0<r<\delta(X)/4$, where
\begin{equation}\label{eqn:geometry}
\begin{split}
 d_p(X,Y)&:=\max\{|x-y|,\sqrt{|t-s|}\},\\
 Q(X,r)&:=B(x,r)\times(t-r^2,t),\\
 \delta(X)&:=\min\{\dist(x,\partial D),\sqrt t\}.
\end{split}
\end{equation}
{
These cylinders are strictly interior: their radii are controlled by the parabolic distance to the initial and lateral boundary.
Condition~\eqref{eqn:morrey} gives interior H\"older regularity through local heat-potential and homogeneous oscillation estimates.
Conversely, for nonnegative $\nu$, a global H\"older bound for a solution implies a growth estimate of the form \eqref{eqn:morrey} with a positive exponent.
For the elliptic counterpart, see \cite{MR998128}; we give the parabolic necessity argument in Lemma~\ref{lem:necessity}.
}

{
For a homogeneous equation, extending interior H\"older regularity to the boundary depends on the geometry of the complement of $D$.
Classical estimates use an exterior volume density condition; see \cite{LSU,Lieberman1996}.
Capacitary approaches allow substantially more general boundaries.
A thermal-capacity Wiener criterion for the heat equation appears in \cite{EvansGariepy1982}, while results of Ziemer \cite{MR563383} and subsequent work \cite{GariepyZiemer1982} concern boundary continuity for more general coefficients.
Quantitative capacitary estimates are also available; see \cite{BiroliMosco1985}.
Here we impose the spatial capacity density condition (CDC): there is $\gamma>0$ such that
}
\begin{equation}\label{eqn:CDC}
 \capacity(\overline{B(\xi,r)}\setminus D,B(\xi,2r))
 \ge\gamma\capacity(\overline{B(\xi,r)},B(\xi,2r))
\end{equation}
{
for every $\xi\in\partial D$ and $r>0$.
The CDC gives uniform H\"older decay at the lateral boundary for homogeneous solutions with zero Dirichlet values; the initial face is controlled by the corresponding initial-time estimates.
Its geometric significance is already apparent in the elliptic theory: on bounded regular domains, preservation of H\"older regularity by harmonic extension with some positive exponent implies uniform fatness of the complement \cite{MR1924196}.
For cylindrical domains satisfying the spatial CDC and with bounded distance to the spatial boundary, the global H\"older estimate for compatible H\"older initial and lateral boundary data follows by an elementary argument; see \cite[Theorems~6.1 and~6.2]{haraParabolic}.
For possibly time-dependent domains, a related sufficient criterion formulated using time-backwards capacity density is given in \cite[Theorem~1.7 and Corollary~1.11]{HPHK2025}.

These boundary estimates suggest a more specific question for \eqref{eqn:intro-equation}.
Under the CDC, does the interior Morrey condition \eqref{eqn:morrey} suffice for global H\"older regularity, or must the same growth bound also hold on cylinders meeting the parabolic boundary?

An example already covered by existing theories illustrates the difference between the interior Morrey condition \eqref{eqn:morrey} and the requirement that the same growth bound hold also on cylinders meeting the parabolic boundary.
The distance-weighted data below can be treated within existing approaches based on weighted estimates and Dirichlet heat-kernel estimates; see \cite{Michael1977,Cho2010} and \cite{Davies1987,Zhang2002,Hirata2013} for these approaches.
For example, let $D$ be bounded and smooth, let $T<\infty$, and consider the locally finite measure
\[
 d\nu(X)=\delta(X)^{\beta-2}\,dX,\qquad 0<\beta<1.
\]
Since $\delta$ is comparable to $\delta(X)$ on an admissible cylinder $Q(X,r)$, this datum satisfies \eqref{eqn:morrey}.
Nevertheless, its total variation is infinite in every neighborhood of a nontrivial open portion of the lateral boundary away from time zero.
In particular, it cannot be extended across such a boundary portion as a locally finite Radon measure.
Thus the same Morrey growth bound on cylinders meeting the boundary would exclude data allowed by \eqref{eqn:morrey}.

}

{
The classical interior regularity theory and the barrier method complement each other.
We show that their combination yields global H\"older regularity under the same growth condition required in the interior, provided \eqref{eqn:CDC} holds and the source has finite support depth.
Our approach uses the strong barrier constructed for the principal parabolic operator in \cite{haraParabolic}.
Following Ancona's elliptic construction \cite{MR856511}, the barrier exploits the CDC to produce, for a suitable positive exponent $\alpha_{\mathrm{ss}}$, a supersolution $S$ satisfying
\[
 S\asymp\delta^{\alpha_{\mathrm{ss}}},\qquad
 \mathcal HS\ge c\delta^{\alpha_{\mathrm{ss}}-2}.
\]
The comparison principle then combines this boundary decay with the local potential estimate obtained from \eqref{eqn:morrey}.
This construction requires no regularity of the boundary beyond the CDC and no integrability of $\nu$ across the boundary.
The elliptic origins and related applications of strong barriers are discussed in \cite{HaraStrongBarrier,HaraEllipticMorrey}; parabolic Hardy inequalities in related settings were studied in \cite{SeoHeat}.

We prove that the interior growth condition \eqref{eqn:morrey} suffices for a bounded globally H\"older continuous solution, provided \eqref{eqn:CDC} holds and the source has finite support depth,
}
\begin{equation}\label{eqn:support-depth}
 \|\delta\|_{L^\infty(\spt\nu)}<\infty.
\end{equation}
{
The support is understood relative to $D\times(0,T)$.
The latter requirement is automatic when $T<\infty$ or the distance to $\partial D$ is bounded on $D$; it is a global scale condition, not a restriction on the accumulation of mass near the boundary.
The principal result is as follows.
}

\begin{theorem}\label{thm:main}
Assume \eqref{eqn:A}, \eqref{eqn:morrey}, \eqref{eqn:CDC}, and \eqref{eqn:support-depth}.
Then there exists a unique bounded weak solution
\[
 u\in V_{\loc}(D\times(0,T))\cap C(\overline{D\times(0,T)})
\]
to \eqref{eqn:intro-equation}.
Moreover,
\begin{equation}\label{eqn:main-estimate}
 |u(X)|
 \le
 CM\|\delta\|_{L^\infty(\spt\nu)}^{\beta-\ass}
 \min\left\{\delta(X),\|\delta\|_{L^\infty(\spt\nu)}\right\}^{\ass}
\end{equation}
and
\begin{equation}\label{eqn:hoelder-estimate}
 |u(X)-u(Y)|\le CM\|\delta\|_{L^\infty(\spt\nu)}^{\beta-\agh}
 d_p(X,Y)^{\agh}
\end{equation}
for all points $X$ and $Y$ in $\overline{D\times(0,T)}$, where
{$0<\agh\le\ass<\beta$ and $C$, $\ass$, and $\agh$ depend only on $n$, $L$, $\gamma$, and $\beta$.}
{For $\nu=0$, the solution is zero and the estimates are understood accordingly.}
\end{theorem}

\begin{corollary}\label{cor:characterization}
Let $\nu$ be a nonnegative locally finite Radon measure. Assume \eqref{eqn:A}, \eqref{eqn:CDC}, and \eqref{eqn:support-depth}.
The following are equivalent:
\begin{enumerate}[label=\textup{(\roman*)}]
\item Condition \eqref{eqn:morrey} holds for some $0<\beta\le1$ and $M<\infty$.
\item Problem \eqref{eqn:intro-equation} has a bounded globally H\"older continuous weak solution with some positive exponent.
\end{enumerate}
\end{corollary}

{
The corollary is a qualitative equivalence for nonnegative measures: neither direction requires the exponent in the growth condition to coincide with the H\"older exponent of the solution.
For signed measures, the theorem gives sufficiency, but no corresponding necessity statement is asserted for the total variation.
The CDC itself admits a separate quantitative characterization.
Section~\ref{sec:cdc-necessity} shows that, for a fixed principal operator, uniform boundary decay estimates for all admissible nonnegative sources and all finite time horizons imply \eqref{eqn:CDC}.
This is stronger than merely requiring the existence of a H\"older solution for each source.
}

\subsection*{Organization of the paper}
{
Section~\ref{sec:preliminaries} records the local potential estimates, boundary estimates, and strong barrier used in the proof.
Section~\ref{sec:principal-proof} proves Theorem~\ref{thm:main} and Corollary~\ref{cor:characterization}, develops estimates for divergence forcing, and presents examples and counterexamples.
Section~\ref{sec:cdc-necessity} gives the quantitative characterization of the CDC.
Section~\ref{sec:boundary} introduces the measure spaces and normalized H\"older norms and treats compatible nonzero boundary data and mixed forcing.
Section~\ref{sec:lower-order} handles lower-order terms by short-time inversion and continuation; related elliptic problems were studied in \cite{HaraEllipticLower}.
Section~\ref{sec:periodic} concerns periodic homogenization and establishes quantitative global H\"older convergence; an elliptic counterpart is discussed in \cite{HaraEllipticHomogenization}.
}

\subsection*{Notation}
The parabolic boundary consists of the initial and lateral faces; the upper time face is excluded. For $E\subset\R^{n+1}$, set
\[
 [u]_{C^{\alpha,\alpha/2}(E)}
 :=\sup_{\substack{X,Y\in E\\X\ne Y}}
 \frac{|u(X)-u(Y)|}{d_p(X,Y)^\alpha}.
\]
Until Section~\ref{sec:boundary}, we use this seminorm and the supremum norm separately. Boundary values are attained continuously. When $T=\infty$, the closure contains only finite-time points. The letter $C$ denotes a positive constant whose value may change. {For nonnegative quantities $a$ and $b$, the notation $a\asymp b$ means that $cb\le a\le Cb$ for positive constants $c$ and $C$ independent of the variables or limiting parameters under consideration.} Vague convergence means convergence against continuous compactly supported functions.

\section{Preliminaries}\label{sec:preliminaries}
\subsection{Capacity and weak solutions}
For an open set $U\subset\R^n$, let $H_0^1(U)$ be the closure of $C_c^\infty(U)$ in $H^1(U)$. For compact $K\subset U$, define
\begin{equation}\label{eqn:relative-capacity}
 \capacity(K,U)=\inf\left\{\int_U|\nabla\varphi|^2\,dx:
 \varphi\in C_c^\infty(U),\ \varphi\ge1\text{ on }K\right\}.
\end{equation}
Condition \eqref{eqn:CDC} is uniform $2$-fatness of the complement; see \cite{Lewis1988,HKM2006}. An exterior corkscrew condition implies it.

For a cylinder $U\times(a,b)$, put
\[
 V(U\times(a,b))=L^\infty((a,b);L^2(U))\cap L^2((a,b);H^1(U)).
\]
Membership in $V_{\loc}$ means membership in this space on every relatively compact cylinder. A weak solution of $\mathcal Hu=\nu$ is a function $u\in V_{\loc}(D\times(0,T))$ such that
\begin{equation}\label{eqn:weak-formulation}
{
 -\iint_{D\times(0,T)}u\partial_t\varphi\,dX+
 \iint_{D\times(0,T)}A\nabla u\cdot\nabla\varphi\,dX
 =\int_{D\times(0,T)}\varphi\,d\nu
}
\end{equation}
for every $\varphi\in C_c^\infty(D\times(0,T))$.
Supersolutions satisfy the inequality with $\ge$ for nonnegative test functions; subsolutions satisfy the reverse inequality.

We use local energy inequalities for bounded continuous solutions. For example, if $\mathcal Hv={\nu}$ on a neighborhood of $Q(X,2r)$, then
\begin{equation}\label{eqn:local-energy}
 \iint_{Q(X,r)}|\nabla v|^2\,dY
 \le Cr^{-2}\iint_{Q(X,2r)}|v|^2\,dY
       +C\|v\|_{L^\infty(Q(X,2r))}{|\nu|}(Q(X,2r)).
\end{equation}
The constant depends only on {$n$ and $L$}. The cutoffs vanish on the lower and lateral faces, and the terminal term is nonnegative and is discarded. For continuous measure-data solutions, the calculation is justified by regularizing the equation on compact subcylinders: the continuous factors converge uniformly in the measure pairing, while their spatial gradients converge in $L^2$. Terminal times can first be chosen outside a null set and then approached from below. For the solutions constructed in Section~\ref{sec:principal-proof}, the same inequality also follows directly from the smooth-density approximations. These observations apply to differences and to the sign-sensitive calculation in Lemma~\ref{lem:necessity}.

\begin{lemma}\label{lem:std-comparison}
Let $U$ be open and $a<b<\infty$. Let {$u$ and $v$ be bounded elements of $V_{\loc}(U\times(a,b))\cap {C}(\overline{U\times(a,b)})$ and} satisfy $\mathcal Hu\le\mathcal Hv$ weakly. If $u\le v$ on the parabolic boundary, then $u\le v$ throughout the cylinder.
\end{lemma}
{
\begin{proof}
Fix $\epsilon>0$ and use the positive truncation $(u-v-\epsilon)_+$. On each bounded spatial set, continuity and the boundary inequality imply that this truncation vanishes near the initial and lateral faces. After a compact spatial cutoff, it therefore has zero lateral trace and zero initial value. Time regularization then justifies testing on the resulting bounded cylinder; the spatial cutoff is removed by the weighted estimate below. If $U$ is unbounded, use spatial cutoffs $\eta_j$ on $B(0,2j)$, equal to one on $B(0,j)$, and the weight
${w}(x)=\exp(-\sqrt{1+|x|^2})$.
Since $|\nabla{w}|\le{w}$, the weighted energy inequality bounds the gradient cross term by half the energy plus a constant times the weighted square of the truncation. The additional cutoff error is at most
\[
 Cj^{-2}{\|u-v\|_{L^\infty(U\times(a,b))}}^2\int_{|x|>j}{w}(x)\,dx.
\]
{
It tends to zero. After letting $j\to\infty$, the weighted energy inequality reads
\[
 \int_U (u(x,t)-v(x,t)-\epsilon)_+^2w(x)\,dx
 \le C\int_a^t\int_U (u(x,s)-v(x,s)-\epsilon)_+^2w(x)\,dx\,ds.
\]
The initial term is zero. {Gronwall's inequality} therefore gives $(u-v-\epsilon)_+=0$. Letting $\epsilon\downarrow0$ proves $u\le v$.
}
\end{proof}
}

\subsection{Local estimates}
The De Giorgi or Moser iteration gives local $L^2$--$L^\infty$ estimates for weak solutions of the homogeneous equation and its adjoint; see \cite{Moser1964}. Together with energy estimates and the reproducing property of the fundamental solution, these local estimates yield the Gaussian upper bound; see \cite[Section~7]{Aronson1968}. Extend $A$ by the identity outside its original domain. Its fundamental solution satisfies the Aronson estimate \cite{Aronson1967}
\begin{equation}\label{eqn:aronson}
 0\le p(x,t;y,s)\le C(t-s)^{-n/2}
 \exp\left(-\frac{|x-y|^2}{C(t-s)}\right)
 \le C d_p((x,t),(y,s))^{-n}
\end{equation}
for every $s<t$. The last bound holds both when $|x-y|\le\sqrt{t-s}$ and when $|x-y|>\sqrt{t-s}$: in the latter case use the boundedness of $a^{n/2}e^{-a/C}$ for $a>0$. Dirichlet kernels on subdomains are bounded above by this kernel.

{
For the regularity theory with Morrey data, see \cite[Theorem~6.29]{Lieberman1996}. We recall the local heat-potential bound in the form used below.
\begin{lemma}\label{lem:local-potential}
Let $0<\beta\le1$, let $U$ be bounded and open with $\operatorname{diam}U\le R$, and let $b-a\le R^2$. Suppose the zero extensions of $f\in L^\infty(U\times(a,b))$ and ${\bm{F}}\in L^2(U\times(a,b);\R^n)$ satisfy
\[
 \|f\|_{L^1(Q(Y,r))}\le Mr^{n+\beta},\qquad
 \|{\bm{F}}\|_{L^2(Q(Y,r))}\le M'r^{n/2+\beta}
\]
for all $Y\in\R^{n+1}$ and $r>0$. {The solution of $\mathcal Hw=f+\divergence\bm{F}$ with zero initial and Dirichlet boundary values satisfies}
\begin{equation}\label{eqn:local-potential}
 \|w\|_{L^\infty(U\times(a,b))}\le C(M+M')R^\beta,
\end{equation}
where $C$ depends only on {$n$, $L$, and $\beta$}. For a local cylinder $Q(X,R)$ with $R\le\delta(X)/16$, it suffices that the two growth bounds hold on admissible cylinders in the ambient domain.
\end{lemma}
{
For $X=(x,t)$, the estimate follows from \eqref{eqn:aronson} and the standard adjoint Caccioppoli bound
}
\begin{equation}\label{eqn:local-kernel-gradient}
 \|\nabla_y p_U(X,\cdot)\|_{L^2((Q(X,r)\setminus Q(X,r/2))\cap(U\times(a,t)))}
 {
\le Cr^{-n/2}.
}
\end{equation}
Apply the kernel representation and sum over concentric annuli. {
To verify the local version, restrict either datum to $Q(X_0,R)$, where $R\le\delta(X_0)/16$, and extend it by zero. For $r\le R$, if $Q(Y,r)$ meets this cylinder, choose $z$ in the intersection of their spatial balls and set $\widehat s=\min\{s,t_0\}$, where $Y=(y,s)$ and $X_0=(x_0,t_0)$. Then
\[
 Q(Y,r)\cap Q(X_0,R)\subset Q((z,\widehat s),2r),\qquad
 \delta(z,\widehat s)\ge\delta(X_0)-R>8R.
\]
The center $(z,\widehat s)$ is interior, so the original growth bound applies at radius $2r$. For $r>R$, use the total mass or $L^2$ norm on $Q(X_0,R)$. Hence restriction and zero extension preserve either growth bound with a constant depending only on {$n$ and $\beta$}. The same argument applies to a locally finite measure, with total variation in place of the $L^1$ norm. For a terminal face at $T$, take limits from below.
} {When $\bm{F}=0$, boundedness of $f$ gives a continuous representative of $w$ at the upper center.}
}

The weak Harnack inequality for nonnegative supersolutions and the resulting homogeneous oscillation estimate are given in \cite[Chapter~VI]{Lieberman1996}; see also \cite{Moser1964,Moser1971}.
Choose $\ain$, depending only on {$n$, $L$, and $\beta$}, so that
\[
 0<\ain<\min\{\beta,1\}
\]
and $\ain$ is strictly smaller than an exponent in the homogeneous oscillation estimate.

\begin{lemma}\label{lem:interior-forced}
Let $X_0\in D\times(0,T)$ and $0<R\le\delta(X_0)/16$.
{Suppose that $u$ is a locally bounded weak solution of $\mathcal Hu=f+\divergence\bm{F}$, where $f$ is bounded and $\bm{F}\in L^2_{\mathrm{loc}}$, and that
\[
 \|f\|_{L^1(Q(Y,r))}\le Mr^{n+\beta},\qquad
 \|\bm{F}\|_{L^2(Q(Y,r))}\le M'r^{n/2+\beta}
\]
for $Y\in D\times(0,T)$ and $0<r<\delta(Y)/4$. Then $u$ has a continuous representative satisfying}
\begin{equation}\label{eqn:interior-forced}
 \osc_{Q(X_0,r)}u\le C(r/R)^{\ain}
 \left(\osc_{Q(X_0,R)}u+{(M+M')R^\beta}\right)
\end{equation}
for $0<r\le R$. The constant depends only on {$n$, $L$, and $\beta$} and is independent of the supremum norm of the forcing.
\end{lemma}
\begin{proof}
{We first work with essential oscillations.} Fix $0<a\le R$ and let $w$ solve the equation with the given forcing in $Q(X_0,a)$ and with $w=0$ on $\partial_pQ(X_0,a)$. {Lemma~\ref{lem:local-potential} gives
\[
 \|w\|_{L^\infty(Q(X_0,a))}\le {C(M+M')a^\beta}.
\]}

The function $h=u-w$ solves the homogeneous equation in $Q(X_0,a)$. {Let $C_0$ be the constant in the homogeneous oscillation estimate for $h$, with its exponent chosen strictly larger than $\ain$.} Choose $0<\theta<1$ so that the homogeneous contraction on $Q(X_0,\theta a)$ is at most $\tfrac12\theta^{\ain}$. It follows that
\[
 \osc_{Q(X_0,\theta a)}u
 \le\tfrac12\theta^{\ain}\osc_{Q(X_0,a)}u+{C(M+M')a^\beta}.
\]
Iterate with $a=\theta^kR$. Since $\ain<\beta$, the resulting geometric series is bounded by a constant times ${(M+M')R^\beta}$. This proves \eqref{eqn:interior-forced} at these radii. Monotonicity of oscillation gives the estimate at intermediate radii. {The decay of essential oscillations gives a continuous representative.} Continuity extends the same estimate to the upper face by taking limits from below.
\end{proof}

\begin{lemma}\label{lem:positive-response}
Let $v$ solve $\mathcal Hv=1$ in $Q(X,R)$ with zero parabolic boundary values. Then
\[
 v(X)\ge cR^2,
\]
where $c>0$ depends only on {$n$ and $L$}.
\end{lemma}
\begin{proof}
Scaling reduces the claim to $X=(0,0)$ and $R=1$. Comparison with the function $s+1$ gives $0\le v\le1$. Choose a fixed nonnegative test function $\varphi$, of positive integral, supported in a cylinder whose closure lies in $B(0,1/2)\times(-3/4,-1/2)$. Choose a slightly larger interior cylinder $E$, still separated from the lower and upper faces. Testing with $\varphi$ and applying the Cauchy--Schwarz inequality gives
\[
 c\le C\iint_E v\,dY+C\left(\iint_{\spt\varphi}|\nabla v|^2\,dY\right)^{1/2}.
\]
{Choose the support of $\varphi$ in a smaller cylinder compactly contained in $E$. Applying \eqref{eqn:local-energy} with a cutoff supported in $E$ and forcing measure $dY$ bounds the last integral by $C\iint_E(v^2+v)\,dY\le C\iint_Ev\,dY$. Hence $\iint_Ev\,dY\ge c$. Choose $E$ within $B(0,1/2)\times(-7/8,-3/8)$. Since $0\le v\le1$, the same lower bound holds for its weak Harnack power when that power is below one. Apply the weak Harnack inequality on a finite chain of overlapping interior cylinders whose later time portions successively reach $B(0,1/8)\times(-1/64,0)$. Their radii and time gaps are fixed independently of $A$. Thus $v(0,0)\ge c$.} The value at time zero is taken from below. Rescaling proves the result.
\end{proof}

\begin{lemma}\label{lem:necessity}
Let $\nu\ge0$ and let $u$ be a continuous weak solution of $\mathcal Hu=\nu$. If $[u]_{C^{\alpha,\alpha/2}(D\times(0,T))}<\infty$ for some $0<\alpha\le1$, then
\begin{equation}\label{eqn:necessity}
 \nu(Q(X,r))\le C[u]_{C^{\alpha,\alpha/2}(D\times(0,T))}r^{n+\alpha}
\end{equation}
for every $X\in D\times(0,T)$ and $0<r<\delta(X)/4$.
The constant depends only on {$n$, $L$, and $\alpha$}.
\end{lemma}
\begin{proof}
On $Q(X,2r)$ set $v=u-\sup_{Q(X,2r)}u$. Thus $v\le0$ and
$|v|\le Cr^\alpha{[u]_{C^{\alpha,\alpha/2}(D\times(0,T))}}$ there.
Choose cutoffs {$\eta$ and $\chi$} with $\eta=1$ on $Q(X,r)$, supported spatially and from below inside $Q(X,3r/2)$, and $\chi=1$ on that cylinder, supported inside $Q(X,2r)$. Their spatial and time derivatives are bounded by $C/r$ and $C/r^2$. Integrate up to a time $\tau<t$. Testing the equation with $v\chi^2$ makes the source term nonpositive. After discarding the nonnegative terminal energy and taking $\tau\uparrow t$, we obtain
\[
 \iint_{Q(X,3r/2)}|\nabla u|^2\,dY
 \le Cr^{-2}\iint_{Q(X,2r)}v^2\,dY
 \le C{[u]_{C^{\alpha,\alpha/2}(D\times(0,T))}}^2r^{n+2\alpha}.
\]
Testing instead with $\eta$ gives a nonpositive terminal term because $v\le0$. Hence, by the Cauchy--Schwarz inequality,
\[
 \nu(Q(X,r))\le Cr^{-2}\iint_{Q(X,3r/2)}|v|\,dY
 +Cr^{n/2}\left(\iint_{Q(X,3r/2)}|\nabla u|^2\,dY\right)^{1/2}.
\]
The two bounds imply \eqref{eqn:necessity}. The regularization described after \eqref{eqn:local-energy} justifies the tests; monotone convergence recovers the open backward cylinder at the terminal time.
\end{proof}

\subsection{Boundary decay and the strong barrier}
{
Strong barriers for weighted quasilinear elliptic equations and their application to boundary-singular Poisson data were constructed in \cite{HaraStrongBarrier}. We use the parabolic construction of \cite{haraParabolic} below.
}
Under {\eqref{eqn:CDC}}, the capacitary boundary estimates of Ziemer \cite{MR563383} give H\"older decay for homogeneous solutions with zero lateral values. The quantitative estimate used in \cite[Lemma~3.9 and Appendix~A]{haraParabolic} is obtained by modifying Ziemer's argument. Using this estimate, \cite[Theorem~4.1]{haraParabolic} extends Ancona's gluing construction of strong barriers \cite{MR856511} to parabolic equations. The zero initial condition permits extension by zero to earlier times, so the boundary estimate also applies near the initial face. Choose $0<\abd<1$, depending only on {$n$, $L$, and $\gamma$}, small enough for both the boundary estimate and the strong-barrier construction.

\begin{proposition}\label{thm:std-barrier}
Assume \eqref{eqn:A} and \eqref{eqn:CDC}. For each $0<\ass<\abd$, there is a nonnegative function
$S\in V_{\loc}(D\times(0,\infty))\cap C(\overline{D\times(0,\infty)})$ such that
\begin{equation}\label{eqn:strong_barrier}
 \delta^{\ass}\le S\le C\delta^{\ass},\qquad
 \mathcal HS\ge c\delta^{\ass-2}
\end{equation}
in $D\times(0,\infty)$, with zero initial and lateral values.
Here $A$ is extended by the identity for $t>T$, and {$c$ and $C$} depend only on {$n$, $L$, $\gamma$, and $\ass$}.
\end{proposition}
Indeed, apply the cited strong-barrier theorem and then take the power $\ass/\abd$. The weak chain rule for an increasing concave power preserves the supersolution inequality; it follows first for regularized powers and then by a local limit.

{
Let $p_D$ denote the Dirichlet heat kernel, obtained by increasing bounded interior exhaustions. A consequence of \eqref{eqn:strong_barrier} is
\begin{equation}\label{eqn:weightedkernel}
 \int_0^t\int_D p_D(X,Y)\delta(Y)^{\ass-2}\,dy\,ds
 \le C\delta(X)^{\ass}.
\end{equation}
Indeed, bounded nonnegative truncations of $\delta^{\ass-2}$ with compact interior support have heat potentials bounded by $c^{-1}S$ by the comparison principle. Comparison applies on each finite time strip because $S$ is bounded there by \eqref{eqn:strong_barrier}. The monotone convergence theorem and \eqref{eqn:strong_barrier} give \eqref{eqn:weightedkernel}. The density in \eqref{eqn:weightedkernel} need not have finite total mass.
}

{For the proof in Section~\ref{sec:principal-proof}, fix $0<\ass<\min\{\beta,\abd,1\}$ and put $\agh=\min\{\ain,\ass\}$.}

\section{{The principal theorem, divergence data, and examples}}\label{sec:principal-proof}
\subsection{Smooth data and the ratio estimate}\label{subsec:boundary-decay}
For $0\le f\in C_c^\infty(D\times(0,T))$, the solution $u$ of \eqref{eqn:intro-equation} with $\nu=f\,dX$ is obtained by solving on bounded interior exhaustions and taking their increasing limit. {The whole-space Gaussian estimate gives a uniform bound for these solutions on every finite time interval.} Since the forcing is bounded and supported in a compact time interval, comparison after that interval gives a uniform bound for all later times. Local energy estimates and interior regularity pass the equation to the limit. Near the parabolic boundary, comparison with a sufficiently large multiple of $S$ gives zero continuous boundary values: on the compact support, $f\le C\delta^{\ass-2}$, and outside it $f=0$. {
More explicitly, if $u_m$ is the solution on the $m$th exhaustion, choose $C_f$ so that $f\le C_f\mathcal HS$ on its compact support. Then comparison gives
\[
 0\le u_m\le C_fS\le CC_f\delta^{\ass}.
\]
The same bound passes to the increasing limit and gives its zero continuous boundary values.
} Uniqueness follows from Lemma~\ref{lem:std-comparison} on finite time intervals.

\begin{proposition}\label{prop:boundary-decay}
{
Let $0\le f\in C_c^\infty(D\times(0,T))$, put $\nu=f\,dX$, and let $u$ be the bounded continuous solution of \eqref{eqn:intro-equation}. If $\nu$ satisfies \eqref{eqn:morrey}, then $u$ satisfies \eqref{eqn:main-estimate}.
}
\end{proposition}
\begin{proof}
{
Assume $f\ne0$. The ratio $u/S$ is continuous on $\spt f$ and has a finite maximum there. For every $k>\max_{\spt f}(u/S)$, the continuous function $(u-kS)_+$ vanishes in a neighborhood of $\spt f$. Outside that support, $u-kS$ is a subsolution, so its positive part is a subsolution on the whole cylinder. This positive part is bounded by $\|u\|_{L^\infty(D\times(0,T))}$, even if $S$ is unbounded. It has zero parabolic boundary values. Lemma~\ref{lem:std-comparison} gives $u\le kS$. Decreasing $k$ to the support maximum and using continuity of $u/S$ gives
\[
 \|u/S\|_{L^\infty(D\times(0,T))}
 =\max_{\spt f}\frac{u}{S}<\infty.
\]
Choose $X_0\in\spt f$ where this maximum is attained and put $r=\delta(X_0)/16$. Let $w$ and $v$ solve $\mathcal Hw=f$ and $\mathcal Hv=1$ in $Q(X_0,r)$, respectively, with $w=v=0$ on $\partial_pQ(X_0,r)$. The function $\|u/S\|_{L^\infty(D\times(0,T))}S-u$ is nonnegative and vanishes at $X_0$. Since $\delta$ is comparable to $\delta(X_0)$ on this cylinder,
\[
 \mathcal H\bigl(\|u/S\|_{L^\infty(D\times(0,T))}S-u+w\bigr)
 =\|u/S\|_{L^\infty(D\times(0,T))}\mathcal HS
 \ge c\|u/S\|_{L^\infty(D\times(0,T))}\delta(X_0)^{\ass-2}.
\]
Comparison with $c\|u/S\|_{L^\infty(D\times(0,T))}\delta(X_0)^{\ass-2}v$ and evaluation at the upper center give
\[
 w(X_0)\ge c\|u/S\|_{L^\infty(D\times(0,T))}\delta(X_0)^{\ass-2}v(X_0).
\]
By Lemmas~\ref{lem:local-potential} and \ref{lem:positive-response},
\[
 \|u/S\|_{L^\infty(D\times(0,T))}
 \le CM\delta(X_0)^{\beta-\ass}
 \le CM\|\delta\|_{L^\infty(\spt f)}^{\beta-\ass}.
\]
Combining this with \eqref{eqn:strong_barrier} proves the bound proportional to $\delta(X)^{\ass}$. On $\spt f$ it also gives
$u\le CM\|\delta\|_{L^\infty(\spt f)}^\beta$.
To extend this constant bound to the whole cylinder, apply the same positive-part argument to $(u-k)_+$ for $k$ larger than the maximum on the support. Taking the smaller of the two global bounds proves the proposition.
}
\end{proof}

\begin{proposition}\label{prop:two-point}
Under the hypotheses of Proposition~\ref{prop:boundary-decay}, the solution satisfies \eqref{eqn:hoelder-estimate}.
\end{proposition}
\begin{proof}
It suffices to consider two distinct interior points $X=(x,t)$ and $Y=(y,s)$, ordered so that $t\ge s$. Put $r=d_p(X,Y)$ and choose
\begin{equation}\label{eqn:two-point-radius}
 R=\min\left\{\frac{\delta(X)}{16},
                  \|\delta\|_{L^\infty(\spt f)}\right\}.
\end{equation}
If $r\ge\|\delta\|_{L^\infty(\spt f)}$, the supremum estimate proves the assertion. We may therefore assume $r<\|\delta\|_{L^\infty(\spt f)}$.

First suppose $2r\ge R$. If $R=\|\delta\|_{L^\infty(\spt f)}$, use the supremum estimate and $r\ge R/2$. Otherwise $R=\delta(X)/16$, so
\[
 \delta(X)\le32r,\qquad\delta(Y)\le\delta(X)+r\le33r.
\]
The boundary estimate gives
\[
 |u(X)-u(Y)|
 \le CM\|\delta\|_{L^\infty(\spt f)}^{\beta-\ass}r^{\ass}
 \le CM\|\delta\|_{L^\infty(\spt f)}^{\beta-\agh}r^{\agh}.
\]
The last inequality uses $\agh\le\ass$ and $r<\|\delta\|_{L^\infty(\spt f)}$.

Now suppose $2r<R$. Since $|x-y|\le r$ and $t-r^2\le s\le t$,
\[
 X,Y\in\overline{Q(X,2r)},\qquad
 Q(X,2r)\subset Q(X,R).
\]
Apply Lemma~\ref{lem:interior-forced} with center $X$, outer radius $R$, and inner radius $2r$. The estimate extends to the upper face by continuity. We obtain
\[
 |u(X)-u(Y)|\le C(r/R)^{\ain}
 \left(\osc_{Q(X,R)}u+MR^\beta\right).
\]
If $R=\delta(X)/16$, then $\delta(Z)\le\delta(X)+R=17R$ for $Z\in Q(X,R)$, and boundary decay bounds this oscillation by
\[
 CM\|\delta\|_{L^\infty(\spt f)}^{\beta-\ass}R^{\ass}.
\]
If $R=\|\delta\|_{L^\infty(\spt f)}$, the supremum estimate gives the same bound. The term $MR^\beta$ also obeys this bound because $R\le\|\delta\|_{L^\infty(\spt f)}$ and $\beta>\ass$. Consequently
\[
 |u(X)-u(Y)|\le
 CM\|\delta\|_{L^\infty(\spt f)}^{\beta-\ass}
 r^{\ain}R^{\ass-\ain}.
\]
If $\ain\le\ass$, use $R\le\|\delta\|_{L^\infty(\spt f)}$; if $\ass<\ain$, use $r<R$. Since $\agh=\min\{\ain,\ass\}$, both cases give the required estimate. Finally pass to the closure by continuity.
\end{proof}

\subsection{Approximation and passage to measures}
\begin{proposition}\label{prop:approximation}
Let $\nu\ge0$ satisfy \eqref{eqn:morrey} and \eqref{eqn:support-depth}. There are $0\le f_j\in C_c^\infty(D\times(0,T))$ such that $f_j\,dX\to\nu$ vaguely,
\begin{equation}\label{eqn:approx-growth}
 \iint_{Q(X,r)}f_j\,dY\le CMr^{n+\beta}
\end{equation}
for all admissible {$X$ and $r$}, and
\begin{equation}\label{eqn:approx-support}
 \|\delta\|_{L^\infty(\spt f_j)}
 \le2\|\delta\|_{L^\infty(\spt\nu)}.
\end{equation}
The constant in \eqref{eqn:approx-growth} depends only on {$n$ and $\beta$}.
\end{proposition}
\begin{proof}
{We control both the support scale and the growth constant of the approximations. Assume $\nu\ne0$. Choose an increasing sequence of compact sets $K_j$ exhausting $D\times(0,T)$ and restrict $\nu$ to $K_j$.} Put
\[
 d_j=\inf_{(y,s)\in K_j}
 \min\{\dist(y,\partial D),\sqrt s,\sqrt{T-s}\},
\]
where the last term is omitted if $T=\infty$. Choose $\epsilon_j>0$ so that
\begin{equation}\label{eqn:approx-epsilon}
 \epsilon_j\le\min\left\{j^{-1},\frac{d_j}{8},
                   \|\delta\|_{L^\infty(\spt\nu)}\right\}.
\end{equation}
Choose $0\le\eta\in C_c^\infty(B(0,1)\times(-1,1))$ with ${
\int_{\R^{n+1}}\eta\,dX
}=1$, and define
\[
 \eta_{\epsilon_j}(x,t)
 =\epsilon_j^{-n-2}\eta(x/\epsilon_j,t/\epsilon_j^2),
 \qquad f_j=\eta_{\epsilon_j}*({\bm{1}}_{K_j}\nu),
\]
extending the restricted measure by zero outside its domain. The choice $d_j\ge8\epsilon_j$ keeps the convolution support inside $D\times(0,T)$. Since $\delta$ is $1$-Lipschitz in $d_p$,
\[
 \sup_{\spt f_j}\delta
 \le\sup_{\spt\nu}\delta+\epsilon_j
 \le2\sup_{\spt\nu}\delta.
\]
This proves \eqref{eqn:approx-support}. Approximate-identity convergence on each compact subset proves vague convergence.

Fix an admissible $Q(X,r)$, where $X=(x,t)$. First suppose $r\ge\epsilon_j$. For a translation $(a,b)$ in the support of $\eta_{\epsilon_j}$, we have $|a|\le\epsilon_j\le r$ and $|b|\le\epsilon_j^2\le r^2$. The translated cylinder has spatial ball $B(x-a,r)$ and time interval $(t-b-r^2,t-b)$. Intersect the latter with $(0,T)$; if the intersection is empty there is nothing to estimate.
Cover $B(x-a,r)$ by at most $5^n$ balls of radius $r/2$ whose centers belong to that ball. Such a covering is obtained from a maximal $r/2$-separated set; its disjoint $r/4$-balls lie in $B(x-a,5r/4)$. Partition the intersected time interval into five equal intervals, each of length at most $r^2/5<r^2/4$. Use the upper endpoint of each interval as the upper time of a backward cylinder of radius $r/2$. The strict inequality ensures that the cylinders overlap at the subdivision times, so no intermediate time slice is omitted.

For each resulting center $Y=(y,s)$,
\[
 |x-y|\le2r,\qquad s\ge t-2r^2,
\]
and therefore
\begin{equation}\label{eqn:translated-center-distance}
 \dist(y,\partial D)>2r,\qquad
 \sqrt s\ge\sqrt t-\sqrt2\,r>2r.
\end{equation}
Since $\delta(X)>4r$, we have $\delta(Y)>2r$ and
\[
 \frac r2<\frac{\delta(Y)}4.
\]
When an upper endpoint equals $T$, the bound is obtained by approaching it from below and applying {Fatou's lemma}; the measure is defined only for times less than $T$. The growth assumption on these at most $5^{n+1}$ cylinders bounds the mass of every translated cylinder by $CMr^{n+\beta}$. Averaging with respect to the convolution kernel proves \eqref{eqn:approx-growth} for $r\ge\epsilon_j$.

For $r<\epsilon_j$, estimate the density first. If $f_j(Z)\ne0$, there is a point of $K_j\cap\spt\nu$ at parabolic distance at most $\epsilon_j$ from $Z=(z,\tau)$. Consequently
\[
 \delta(Z)\ge d_j-\epsilon_j\ge7\epsilon_j>4\epsilon_j.
\]
If $T<\infty$, also $\tau+\epsilon_j^2<T$, since the original support time is at most $T-d_j^2$ and the two time displacements total at most $2\epsilon_j^2<d_j^2$. The box of source points entering the convolution at $Z$ is the union of
\[
 Q(Z,\epsilon_j),\qquad
 Q((z,\tau+\epsilon_j^2),\epsilon_j),
\]
apart from the intervening time slice. Increase the first upper time by $h>0$ and let $h\downarrow0$. Every point of the open source box, including that time slice, then belongs to the union for all sufficiently small $h$. Both radii remain admissible, because increasing the time does not decrease $\delta$ at the fixed spatial point. The Fatou lemma therefore bounds the source mass in the box by $2M\epsilon_j^{n+\beta}$, and
\[
 0\le f_j\le C M\epsilon_j^{\beta-2}.
\]
It follows that
\[
 \iint_{Q(X,r)}f_j\,dY
 \le CM\epsilon_j^{\beta-2}r^{n+2}
 \le CMr^{n+\beta},
\]
since $r<\epsilon_j$ and $\beta<2$. This completes the proof.
\end{proof}

\begin{proof}[Proof of Theorem~\ref{thm:main}]
First let $\nu\ge0$. Take $f_j$ from Proposition~\ref{prop:approximation} and let $u_j$ solve \eqref{eqn:intro-equation} with $\nu=f_j\,dX$. Propositions~\ref{prop:boundary-decay} and \ref{prop:two-point} give uniform boundedness, a uniform global H\"older seminorm, and uniform decay to zero on the parabolic boundary. A diagonal application of the Arzel\`a--Ascoli theorem yields local uniform convergence on the closure to a bounded continuous function $u$. The two estimates pass to the limit, with a changed structural constant.

{
{Fix a cylinder $Q\Subset D\times(0,T)$ and choose a larger cylinder} $Q'\Subset D\times(0,T)$ containing its closure. The local energy estimate, the uniform supremum bound, and \eqref{eqn:approx-growth} give
\[
 \sup_j\left(\|u_j\|_{L^\infty(Q')}
 +\|\nabla u_j\|_{L^2(Q)}\right)<\infty.
\]
Here the masses on $Q'$ are uniformly bounded by a finite admissible covering. After diagonal subsequence extraction,
\[
 u_j\longrightarrow u\quad\text{locally uniformly},\qquad
 \nabla u_j\rightharpoonup\nabla u\quad\text{locally in }L^2.
\]
The spatial weak limit is identified by integration against compactly supported smooth functions. For every $\varphi\in C_c^\infty(D\times(0,T))$, these convergences and vague convergence give
\[
 -\iint_{D\times(0,T)}u\partial_t\varphi\,dX
 +\iint_{D\times(0,T)}A\nabla u\cdot\nabla\varphi\,dX
 =\lim_j\iint_{D\times(0,T)}f_j\varphi\,dX
 =\int_{D\times(0,T)}\varphi\,d\nu.
\]
Thus $\mathcal Hu=\nu$ and $u\in V_{\loc}$. The uniform boundary decay passes to $u$ and gives zero continuous values also at the intersection of the initial and lateral faces.
}

For signed $\nu$, apply the construction to its Jordan components {$\nu_+$ and $\nu_-$}. {Both obey the same growth bound, and $\operatorname{supp}\nu_\pm\subset\operatorname{supp}\nu$, so the supremum of $\delta$ on either support is at most $\|\delta\|_{L^\infty(\operatorname{supp}\nu)}$.} Subtract the resulting solutions. Since $\beta-\ass>0$ and $\beta-\agh>0$, their estimates imply the asserted estimates with the support scale of $\nu$. Finally, the difference of any two bounded continuous weak solutions satisfies the homogeneous equation with zero parabolic values. Lemma~\ref{lem:std-comparison}, applied on each finite time interval, proves uniqueness.
\end{proof}

{
\begin{remark}[Supersolutions for measure data]
For a nonnegative datum satisfying \eqref{eqn:morrey} and \eqref{eqn:support-depth}, a supersolution dominating the measure can also be constructed during the approximation. Let $f_j$ be as in Proposition~\ref{prop:approximation} and solve the smooth problems by the exhaustion at the start of this section. Proposition~\ref{prop:boundary-decay} bounds these solutions $u_j$ by $CM\|\delta\|_{L^\infty(\spt\nu)}^{\beta-\ass}S$. Consequently $S+u_j$ are uniformly comparable to $S$ and satisfy
\[
 \mathcal H(S+u_j)\ge f_j+c\delta^{\ass-2}.
\]
The compactness and weak-limit argument in the proof of Theorem~\ref{thm:main} passes this inequality to a supersolution with right-hand side $\nu+c\delta^{\ass-2}$. This construction uses the smooth approximation and ratio estimate directly; it does not assume measure-data solvability as an input.
\end{remark}

}
\begin{proof}[Proof of Corollary~\ref{cor:characterization}]
{Theorem~\ref{thm:main}, under \eqref{eqn:support-depth}, proves sufficiency.} Lemma~\ref{lem:necessity} proves necessity with the H\"older exponent of the given solution.
\end{proof}

{
\subsection{Divergence data}\label{subsec:divergence-data}
{In this subsection, assume \eqref{eqn:A} and \eqref{eqn:CDC}, and fix
$0<\ass<\min\{\abd,1\}$.}
We construct solutions for divergence data {by the comparison principle} and energy approximation.

For a distribution $\sigma$, let $\spt\sigma$ denote its support relative to $D\times(0,T)$: its complement is the largest open set where $\sigma$ vanishes. We use the supremum of $\delta$ on this support, with value zero for the zero distribution. In particular, for $\sigma=\nu+\divergence\bm{F}$ this is the support of the sum as a distribution, not the union of the supports of its representation components.

\begin{theorem}\label{thm:divergence}
Suppose $\bm{F}\in L^2_{\mathrm{loc}}(D\times(0,T);\R^n)$ satisfies
\begin{equation}\label{eqn:data}
 \|\bm{F}\|_{L^2(Q(Y,r))}\le M'r^{n/2+\ass},
 \qquad 0<r<\delta(Y)/4.
\end{equation}
Assume
\begin{equation}\label{eqn:divergence-support-depth}
 \|\delta\|_{L^\infty(\spt(\divergence\bm{F}))}<\infty.
\end{equation}
There is a unique bounded continuous weak solution $v\in V_{\mathrm{loc}}(D\times(0,T))$ of $\mathcal Hv=\divergence\bm{F}$ with zero initial and lateral values. It satisfies
\begin{equation}\label{eqn:divergence-bound}
 |v(X)|\le CM'\min\left\{\delta(X),
       \|\delta\|_{L^\infty(\spt(\divergence\bm{F}))}\right\}^{\ass}
\end{equation}
for every $X\in D\times(0,T)$. The constant depends only on {$n$, $L$, $\gamma$, and $\ass$}. Condition \eqref{eqn:divergence-support-depth} is automatic for finite $T$.
\end{theorem}
\begin{proof}
We first construct a solution, bounded on each finite time strip, and prove
\begin{equation}\label{eqn:divergence-pointwise}
 |v(X)|\le CM'\delta(X)^{\ass}.
\end{equation}
First let $\bm{F}\in C_c^\infty(D\times(0,T);\R^n)$. The smooth scalar source $\divergence\bm{F}$ gives a bounded continuous solution {by the smooth-forcing construction at the beginning of Subsection~\ref{subsec:boundary-decay}.} Apply the comparison argument in Proposition~\ref{prop:boundary-decay} to $v$ and $-v$, using $\spt(\divergence\bm{F})$ for the support maximum. The only change is the estimate for the local response: Lemma~\ref{lem:local-potential}, with $f=0$, gives
\[
 \|w\|_{L^\infty(Q(X_0,r))}\le CM'r^{\ass}
\]
when $\mathcal Hw=\divergence\bm{F}$ in $Q(X_0,r)$ and $w=0$ on its parabolic boundary. Hence the same comparison, with the forcing exponent equal to $\ass$, gives \eqref{eqn:divergence-pointwise}. This also applies when the growth condition is imposed only for $r<\delta(Y)/8$, by reducing the fixed radius in that comparison.

Next suppose that $\bm{F}\in L^2(D\times(0,T);\R^n)$ has compact support. For this $\bm{F}$, extend it by zero and set $\bm{F}_\epsilon=\eta_\epsilon*\bm{F}$, where $\eta_\epsilon$ is nonnegative, has integral one, and is supported in $|z|<\epsilon$, $0<\tau<\epsilon^2$. {Let $d>0$ be smaller than the spatial distance of $\operatorname{supp}\bm{F}$ from $\partial D$ and the square roots of its distances from the finite time endpoints. Choose $0<\epsilon<d/16$. These fields then have a common compact support in $D\times(0,T)$ and converge to $\bm{F}$ in $L^2(D\times(0,T))$.} For $r<\delta(Y)/8$, any translated cylinder meeting $\spt\bm{F}$ has an interior center and satisfies $r<\delta(Y-Z)/4$ when $Z\in\spt\eta_\epsilon$. Indeed, the support has positive distance from the parabolic boundary, and $\delta$ is Lipschitz in the parabolic metric. The one-sided time support also gives $s-\tau<T$. Thus {Minkowski's integral inequality} gives
\[
 \|\bm{F}_\epsilon\|_{L^2(Q(Y,r))}
 \le\int_{\R^{n+1}}\eta_\epsilon(Z)
       \|\bm{F}\|_{L^2(Q(Y-Z,r))}\,dZ
 \le M'r^{n/2+\ass}.
\]
{Let $v$ be the energy solution with forcing $\operatorname{div}\bm{F}$ and zero initial and Dirichlet boundary values, and let $u_\epsilon$ be the corresponding solutions for $\bm{F}_\epsilon$. Testing the equation for $u_\epsilon-v$} with $u_\epsilon-v$ gives, for each $0<b<T$,
\[
 \sup_{0<t<b}\int_D|u_\epsilon-v|^2\,dx
 +\int_0^b\int_D|\nabla u_\epsilon-\nabla v|^2\,dx\,dt
 \le C\int_0^b\int_D|\bm{F}_\epsilon-\bm{F}|^2\,dx\,dt\longrightarrow0.
\]
Therefore \eqref{eqn:divergence-pointwise} passes to $v$ almost everywhere.

For general $\bm{F}$, choose increasing compact sets $K_m\Subset D\times(0,T)$ whose interiors exhaust $D\times(0,T)$, and put $\bm{F}_m=\bm{1}_{K_m}\bm{F}$. Then
\[
 \|\bm{F}_m\|_{L^2(Q(Y,r))}\le M'r^{n/2+\ass},
 \qquad \bm{F}_m\longrightarrow\bm{F}\quad\text{in }L^2_{\mathrm{loc}}(D\times(0,T)).
\]
Let $v_m$ be the energy solutions for $\bm{F}_m$. The compact-support case gives \eqref{eqn:divergence-pointwise} for each $v_m$, with the same constant. For every finite $b<T$, the bound $\delta\le\sqrt b$ makes these solutions uniformly bounded on $D\times(0,b)$, independently of $m$.

Lemma~\ref{lem:interior-forced}, applied with $f=0$ and $\beta=\ass$, gives continuous representatives of $v_m$ and uniform interior H\"older estimates. The constants are independent of $m$.
Thus the Arzel\`a--Ascoli theorem and a diagonal extraction give $v_m\to v$ locally uniformly in $D\times(0,T)$. The boundary bound passes to $v$ and yields \eqref{eqn:divergence-pointwise}.

For each $Q(X_0,2r)\Subset D\times(0,T)$, the local energy inequality gives
\[
 \int_{Q(X_0,r)}|\nabla v_m|^2\,dX
 \le Cr^{-2}\int_{Q(X_0,2r)}|v_m|^2\,dX
      +C\int_{Q(X_0,2r)}|\bm{F}_m|^2\,dX.
\]
Hence a further diagonal extraction gives $\nabla v_m\rightharpoonup\nabla v$ in $L^2$ on compact interior cylinders. For every $\varphi\in C_c^\infty(D\times(0,T))$, pass to the limit in
\[
 -\int_{D\times(0,T)}v_m\partial_t\varphi\,dX
 +\int_{D\times(0,T)}A\nabla v_m\cdot\nabla\varphi\,dX
 =-\int_{D\times(0,T)}\bm{F}_m\cdot\nabla\varphi\,dX.
\]
Therefore $\mathcal Hv=\divergence\bm{F}$ weakly. Local boundedness and the gradient bound give $v\in V_{\mathrm{loc}}$. Estimate~\eqref{eqn:divergence-pointwise} gives continuous zero initial and lateral values. At a finite terminal time, \eqref{eqn:interior-forced}, applied with $f=0$ and $\beta=\ass$ and with centers approaching that time, gives a continuous extension from below. The finite-strip solutions agree on overlaps by Lemma~\ref{lem:std-comparison} and therefore define the solution on $(0,T)$.

If $\divergence\bm{F}=0$, the same comparison principle gives $v=0$. Otherwise \eqref{eqn:divergence-pointwise} bounds $|v|$ on $\spt(\divergence\bm{F})$ by
$CM'\|\delta\|_{L^\infty(\spt(\divergence\bm{F}))}^{\ass}$.
For a level $k$ strictly above this bound, continuity makes $(v-k)_+$ vanish near the distributional support of the forcing. It is a subsolution elsewhere, where the equation is homogeneous. Apply the positive-part comparison from Proposition~\ref{prop:boundary-decay} on each finite strip, to $v$ and $-v$, and decrease $k$ to the support bound. Combining this bound with \eqref{eqn:divergence-pointwise} proves \eqref{eqn:divergence-bound} and gives boundedness on the whole cylinder under \eqref{eqn:divergence-support-depth}. Bounded comparison also proves uniqueness.
\end{proof}

}

\begingroup
\subsection{Examples and counterexamples}\label{sec:example}

A basic class of measures satisfying \eqref{eqn:morrey} is $\nu = f(X) \, dX$, where $|f(X)|\le K\delta(X)^{\beta-2}$ and $0 < \beta < 1$.
{
Such densities can fail to be integrable near the lateral boundary. \eqref{eqn:weightedkernel} also controls their heat potentials: for $\ass<\beta$, bound $\delta^{\beta-2}$ by $\|\delta\|_{L^\infty(D\times(0,T))}^{\beta-\ass}\delta^{\ass-2}$.
}

Let {$n=2$, $D=(-1,1)\times(0,1)$, and $0<T<1/16$}. For
$0\le\beta\le1/2$, define $\nu_\beta$ by requiring that, for
every $\varphi\in C_c(D\times(0,T))$,
\[
 \int_{D\times(0,T)}\varphi\,d\nu_\beta
 =\frac12\int_0^T t^{\beta-1}
   \int_{-\sqrt t}^{\sqrt t}\varphi(x_1,t,t)\,dx_1\,dt.
\]
All these measures are supported on the same moving segments
\[
 \{(x_1,x_2): |x_1|<\sqrt t,\ x_2=t\}.
\]
Their closures meet the initial face only at $(x,t)=(0,0)$, and
all have finite total mass. For every $0<\tau<T$,
\[
 \nu_\beta(D\times(0,\tau))
 =\int_0^\tau t^{\beta-1/2}\,dt
 =\frac{\tau^{\beta+1/2}}{\beta+1/2}.
\]
In particular, $\beta=1/2$ gives a source of constant total power.

\smallskip
\noindent\textit{Positive exponents.}
On the support of $\nu_\beta$, one has $\delta(y,s)=s$.
For $r<\delta(X)/4$, the distance $\delta(Y)$ is comparable to
$\delta(X)$ throughout $Q(X,r)$, and each spatial ball intersects
a horizontal segment in length at most $2r$. Thus
\[
 \nu_\beta(Q(X,r))
 \le C\delta(X)^{\beta-1}r^3
 =Cr^{2+\beta}
   \left(\frac{r}{\delta(X)}\right)^{1-\beta}
 \le Cr^{2+\beta}.
\]
For $\beta>0$, let $u_\beta$ be the globally
H\"older continuous solution supplied by Theorem~\ref{thm:main} for $\nu_\beta$. For every $A$ satisfying \eqref{eqn:A},
\[
 |u_\beta(x_1,t,t)|\le Ct^\agh
\]
for $0<t<T$ and $|x_1|<\sqrt t$. This holds although no uniform bound of order $r^{2+\beta}$ holds for the mass on cylinders meeting the initial and lateral boundary, as the following calculation shows.
Indeed, for $0<r<\sqrt T$, all the source emitted before $r^2$ lies
in $B(0,2r)\cap D$, so
\[
 \frac{\nu_\beta((B(0,2r)\cap D)\times(0,r^2))}
      {r^{2+\beta}}
 =\frac{r^{\beta-1}}{\beta+1/2}\longrightarrow\infty.
\]

The same family also shows that the data exponent $\beta$ cannot
in general be exceeded. For $\beta>0$, put $X_t=((0,t),t)$ and
$r=t/8$. On the time interval of $Q(X_t,r)$, one has $s\asymp t$,
and the intersection of the spatial ball with the source segment
has length comparable to $r$. Hence
\[
 \nu_\beta(Q(X_t,r))
 \asymp t^{\beta-1}r^3\asymp r^{2+\beta}.
\]
For every $\beta<\alpha\le1$, the ratio of this mass to
$r^{2+\alpha}$ tends to infinity as $t\downarrow0$.
Lemma~\ref{lem:necessity} therefore excludes global
$C^{\alpha,\alpha/2}$ regularity for such $\alpha$.

\smallskip
\noindent\textit{The endpoint.}
Set $\beta=0$, keeping the same support. At $X_t=((0,t),t)$,
choose $r=t/8$. The segment moves by only $r^2$ during the time
interval of $Q(X_t,r)$, and hence
\[
 c r^2\le\nu_0(Q(X_t,r))\le C r^2.
\]
No interior growth bound of order $r^{2+\alpha}$ holds for any
$\alpha>0$. Lemma~\ref{lem:necessity} rules out a globally
H\"older continuous solution of any positive exponent.

For the heat operator, continuity at the initial corner itself
fails. To see this directly, {suppose that a continuous solution $u$ with zero initial and Dirichlet boundary values existed.} It is nonnegative {by the comparison principle}. Put
$U_t=(-t,t)\times(0,2t)$ and let $p_{U_t}$ be its Dirichlet heat kernel.
After scaling to $(-1,1)\times(0,2)$, the evaluation and source
points remain in fixed compact subsets of the rectangle, and the
time differences belong to $[1,2]$. Continuity and strict positivity
of the heat kernel on these compact sets therefore give
\[
 p_{U_t}((0,t),t;(y_1,s),s)\ge ct^{-2}
\]
whenever $t-2t^2<s<t-t^2$ and $|y_1|<t/2$, for sufficiently small
$t$. These source points lie in $U_t$ and satisfy $|y_1|<\sqrt s$.
Let $v_t$ solve the equation with this restricted source in $U_t\times(t-2t^2,t)$, with $v_t=0$ on its parabolic boundary. The kernel formula for $v_t$ follows
first for smooth approximations of the restricted source and then
by passage to the limit. The source times in the integral below stay
at least $t^2$ below the evaluation time, so the kernel is continuous
and bounded on this compact integration set. {The comparison principle gives}
$u\ge v_t$, and hence
\[
 u(0,t,t)\ge v_t(0,t,t)
 = \int_{t-2t^2}^{t-t^2}\frac{1}{2s}
       \int_{-t/2}^{t/2}
       p_{U_t}((0,t),t;(y_1,s),s)\,dy_1\,ds
 \ge c>0.
\]
This contradicts continuity and the zero initial value at the corner.

\subsubsection{Signed data: failure of necessity}\label{sec:signed-examples}
The nonnegativity assumption in Lemma~\ref{lem:necessity} cannot
be omitted, even for solutions that are Lipschitz continuous with respect to $d_p$ and belong to the global energy space. The following examples concern the heat operator. We first give spatial oscillations at the lateral boundary and in the interior, then temporal oscillations at the initial face and at an interior time.

\begin{example}
Let $D=(-1,1)^{n-1}\times(0,1)$, $n\ge2$, and write $x=(x',x_n)$ with $x'\in\R^{n-1}$. Choose nonnegative cutoffs $\eta\in C^\infty(\overline D)$ and $\chi\in C_c^\infty((0,T))$. {Assume that $\eta=1$ on the intersection of $\overline D$ with a neighborhood of the spatial origin} and that its support is separated from the other faces of $D$. Let $\chi=1$ near some $0<t_0<T$. Set
\[
 u(x,t)=\eta(x)\chi(t)x_n^2\sin(1/x_n),
\]
and extend $u$ by zero to $x_n=0$. Since
\[
 \frac{d}{dx_n}\bigl(x_n^2\sin(1/x_n)\bigr)
 =2x_n\sin(1/x_n)-\cos(1/x_n),
\]
the solution $u$ is Lipschitz continuous and belongs to $L^2((0,T);H_0^1(D))$. Near $(0,t_0)$, its datum is
\[
 \nu=(\partial_tu-\Delta u)\,dX
 =\left(x_n^{-2}\sin(1/x_n)+2x_n^{-1}\cos(1/x_n)
                         -2\sin(1/x_n)\right)dX.
\]
Put $X_r=((0,5r),t_0)$, where $0\in\R^{n-1}$. For sufficiently small $r>0$, $r<\delta(X_r)/4$ and both cutoffs equal one throughout $Q(X_r,r)$. The substitution $a=1/x_n$ gives
\[
 \int_{9r/2}^{11r/2}x_n^{-2}|\sin(1/x_n)|\,dx_n
 =\int_{2/(11r)}^{2/(9r)}|\sin a|\,da
 \ge c r^{-1}.
\]
For $9r/2<x_n<11r/2$, the spatial cross sections of $B((0,5r),r)$ have $(n-1)$-dimensional volume at least $cr^{n-1}$. Thus the absolute integral of the first term in the density over $Q(X_r,r)$ is at least $cr^n$. The absolute integrals of the remaining terms are bounded by $C(r^{n+1}+r^{n+2})$. Consequently,
\[
 |\nu|(Q(X_r,r))\ge cr^n
\]
for sufficiently small $r$. Hence $\nu$ does not satisfy \eqref{eqn:morrey} for any $\beta>0$.
\end{example}

\begin{remark}
{
Necessity for the total variation can fail even when that variation is finite. The following example has a compactly supported Lipschitz solution and an integrable forcing density.
}
Let $n=2$ and let $D\subset\R^2$ be a bounded domain
containing the origin. Choose $0<r_0<1$ such that
$\cl{B(0,r_0)}\subset D$, and nonnegative cutoffs
$\eta\in C_c^\infty(B(0,r_0))$ and
$\chi\in C_c^\infty((0,T))$, equal to one near the origin
and some $0<t_0<T$, respectively. Set
\[
 u(x,t)=\eta(x)\chi(t)\frac{|x|^2\sin(1/|x|)}{\log^2(e/|x|)},
 \qquad u(0,t)=0,
\]
where the spatial expression is used on $B(0,r_0)$ and $u$ is extended by zero outside the support of $\eta$.
On a neighborhood of the spatial origin and for times near $t_0$, the cutoffs equal one. Writing $r=|x|$ there, we obtain
\[
 \nu=\left(\frac{\sin(1/r)}{r^2\log^2(e/r)}
       +O\!\left(\frac1{r\log^2(e/r)}\right)\right)dX.
\]
The absolute integral of the leading term over $B(0,r)$ is comparable to $1/\log(e/r)$, while the remainder has integral at most $Cr/\log^2(e/r)$. Therefore,
\[
 |\nu|(B(0,r)\times(t_0-r^2,t_0))
 \asymp\frac{r^2}{\log(e/r)}.
\]
Hence \eqref{eqn:morrey} fails for every $\beta>0$. Nevertheless $\nu$ has finite total variation, since the displayed spatial density is integrable near the origin. The solution is Lipschitz and has zero initial and lateral values; no Dirac mass occurs at the origin because the radial flux tends to zero.

\end{remark}

\begin{example}[Temporal oscillation at the initial face]\label{ex:temporal-initial}
The failure of necessity also occurs through temporal oscillation at the initial face, with smooth spatial dependence.
Let $D=B(0,1)\subset\R^n$, $n\ge2$, and $0<T<1/2$. Choose nonnegative cutoffs $\eta\in C_c^\infty(D)$ and $\chi\in C^\infty([0,T])$, equal to one near $x=0$ and $t=0$, respectively, with $\chi=0$ near $T$. Set
\[
 u(x,t)=\eta(x)\chi(t)t\sin(1/t),\qquad u(x,0)=0.
\]
The function $v(t)=t\sin(1/t)$ satisfies $|v(t)|\le t$ and $|v'(t)|\le1+t^{-1}$. For $0<s<t<T$, put $h=t-s$. If $h\le s^2$, integration of the derivative bound gives $|v(t)-v(s)|\le Ch/s\le C\sqrt h$. If $h>s^2$, the amplitude bound gives $|v(t)-v(s)|\le t+s\le C\sqrt h$. The same estimate at $s=0$ follows from the amplitude bound. Thus $u\in C^{1,1/2}(\overline{D\times(0,T)})\cap L^2((0,T);H_0^1(D))$, with zero initial and lateral values.

For $t>0$, the density of $\nu=(\partial_tu-\Delta u)\,dX$ is smooth, so $\nu$ is locally finite. Where both cutoffs equal one,
\[
 \nu=\left(\sin(1/t)-t^{-1}\cos(1/t)\right)dX.
\]
For sufficiently small $r>0$, put $X_r=(0,25r^2)$. Then $\delta(X_r)=5r$ and $r<\delta(X_r)/4$. Both cutoffs equal one on $Q(X_r,r)$, so
\[
 |\nu|(Q(X_r,r))
 =|B(0,r)|\int_{24r^2}^{25r^2}
       |\sin(1/t)-t^{-1}\cos(1/t)|\,dt.
\]
The difference between this time integral and $\int_{24r^2}^{25r^2}t^{-1}|\cos(1/t)|\,dt$ is $O(r^2)$. The substitution $a=1/t$ and the positive mean of $|\cos a|$ over each period give
\[
 \int_{24r^2}^{25r^2}t^{-1}|\cos(1/t)|\,dt
 =\int_{1/(25r^2)}^{1/(24r^2)}\frac{|\cos a|}{a}\,da
 \longrightarrow\frac2\pi\log\frac{25}{24}.
\]
Consequently $|\nu|(Q(X_r,r))\asymp r^n$, and \eqref{eqn:morrey} fails for every $\beta>0$. The total variation is infinite near the initial face, since $\int_0^a|v'(t)|\,dt=\infty$ for every $a>0$.

A finite-total-variation version is obtained by replacing $t\sin(1/t)$ with $t\sin(1/t)/\log^2(e/t)$. The same regularity argument applies, and its time derivative is absolutely integrable at zero. On the same admissible cylinders,
\[
 |\nu|(Q(X_r,r))\asymp\frac{r^n}{\log^2(e/r)}.
\]
Thus \eqref{eqn:morrey} still fails for every $\beta>0$.
\end{example}

\begin{example}[Temporal oscillation at an interior time]\label{ex:temporal-interior}
Temporal oscillation can also give finite total variation while violating every positive Morrey growth exponent at a fixed interior point.
Let $D=B(0,1)\subset\R^n$, $n\ge2$, and $0<t_0<T<1/2$. Choose nonnegative cutoffs $\eta\in C_c^\infty(D)$ and $\chi\in C_c^\infty((0,T))$, equal to one near $0$ and $t_0$, respectively. For $0<|\tau|<1/2$, put
\[
 v(\tau)=\frac{\tau\sin(1/\tau)}{\log^2(e/|\tau|)},\qquad v(0)=0,
 \qquad u(x,t)=\eta(x)\chi(t)v(t-t_0).
\]
For $\tau\ne0$,
\[
 v'(\tau)=-\frac{\cos(1/\tau)}{\tau\log^2(e/|\tau|)}
             +O\!\left(\frac1{\log^2(e/|\tau|)}\right).
\]
Hence $v'$ is absolutely integrable at zero and $v$ is absolutely continuous across zero. Its distributional derivative therefore has no atom there. The density of $\nu=(\partial_tu-\Delta u)\,dX$ is integrable on $D\times(0,T)$, so $\nu$ has finite total variation. The bounds $|v(\tau)|\le C|\tau|$ and $|v'(\tau)|\le C/|\tau|$ give the temporal H\"older exponent $1/2$ by the splitting used in Example~\ref{ex:temporal-initial}; points on opposite sides of zero are handled by the amplitude bound. Thus $u\in C^{1,1/2}(\overline{D\times(0,T)})\cap L^2((0,T);H_0^1(D))$, with compact support in the interior and zero initial and lateral values.

Fix $X_0=(0,t_0)$. For all sufficiently small admissible $r>0$, both cutoffs equal one on $Q(X_0,r)$. Since $|v'|$ is even,
\[
 |\nu|(Q(X_0,r))
 =|B(0,r)|\int_0^{r^2}|v'(\tau)|\,d\tau
 \asymp\frac{r^n}{\log(e/r)}.
\]
Indeed, the leading integral becomes
\[
 \int_0^{r^2}\frac{|\cos(1/\tau)|}{\tau\log^2(e/\tau)}\,d\tau
 =\int_{1/r^2}^\infty\frac{|\cos a|}{a\log^2(ea)}\,da
 \asymp\frac1{\log(e/r)},
\]
since $|\cos a|$ has a positive mean over each period; the remainder has integral at most $Cr^2/\log^2(e/r)$. Therefore \eqref{eqn:morrey} fails for every $\beta>0$, even at this fixed interior center. Here the logarithmic factor makes the time derivative locally integrable across the interior time $t_0$; the preceding example places its singularity at the initial face, outside the open cylinder.
\end{example}

\endgroup

\section{Characterization of the capacity density condition}\label{sec:cdc-necessity}
{For a fixed coefficient matrix, the converse below concerns an estimate uniform in the finite time horizon and the source. Existence of a H\"older solution for each datum alone is a weaker assertion.}

\begin{theorem}\label{thm:cdc-characterization}
{Let $D\subset\R^n$, $n\ge2$, be a domain, and fix a coefficient matrix $A$} satisfying \eqref{eqn:A} on $D\times(0,\infty)$. The following are equivalent:
\begin{enumerate}[label=\textup{(\roman*)}]
\item The domain $D$ satisfies \eqref{eqn:CDC}.
\item There are $0<\ass<\beta\le1$ and $C>0$ such that, for every finite $T>0$ and every {nonnegative locally finite measure} $\nu$ satisfying \eqref{eqn:morrey}, the equation $\mathcal Hu=\nu$ has a bounded continuous weak solution with zero initial and Dirichlet boundary values satisfying
\begin{equation}\label{eqn:cdc-operator-bound}
 0\le u(X)\le {
 CM\|\delta\|_{L^\infty(\spt\nu)}^{\beta-\ass}
 \min\left\{\delta(X),\|\delta\|_{L^\infty(\spt\nu)}\right\}^{\ass}
}
\end{equation}
for every $X\in D\times(0,T)$. The constants {$\beta$, $\ass$, and $C$} are independent of {$T$, $\nu$, and $M$}.
\end{enumerate}
In the forward implication the constants can be chosen in terms of {$n$, $L$, and $\gamma$}. In the reverse implication the {constant $\gamma$ in \eqref{eqn:CDC}} is bounded below in terms of {$n$, $L$, $\beta$, $\ass$, and $C$}.
\end{theorem}
\begin{proof}
{Theorem~\ref{thm:main}, with a fixed $\beta$, proves the forward implication: for every finite $T$, $\|\delta\|_{L^\infty(\spt\nu)}\le\sqrt T<\infty$, so \eqref{eqn:support-depth} is automatic and the estimate has a constant independent of $T$.} For the converse, fix $\xi\in\partial D$ and $R>0$, and set
\[
 U=D\cap B(\xi,R),\qquad K=\overline{B(\xi,R)}\setminus D.
\]
Extend $A$ outside $D$ by the identity matrix, preserving \eqref{eqn:A}. Apply condition~\textup{(ii)} on $D\times(0,R^2)$ to the density $f=R^{-2}{\bm{1}}_U$. For every admissible cylinder,
\[
 \iint_{Q(X,r)}f\,dY
 \le CR^{-2}\min\{r^n,R^n\}r^2
 \le CR^{-\beta}r^{n+\beta}.
\]
For $r\ge R$, the last inequality uses $n+\beta\ge2$; for $r\le R$ it uses $\beta\le2$. The support distance is at most $R$. The assumed estimate and $\beta>\ass$ therefore give
\begin{equation}\label{eqn:cdc-local-bound}
 0\le u(x,t)\le C\left(\frac{\dist(x,\partial D)}R\right)^{\ass}.
\end{equation}

Let $v$ and $w$ be the variational solutions of $\mathcal Hv=R^{-2}$ in $B(\xi,R)\times(0,R^2)$ and $\mathcal Hw=R^{-2}$ in $U\times(0,R^2)$, respectively, with zero initial and Dirichlet boundary values. The comparison principle gives $0\le v\le t/R^2\le1$. Applying Lemma~\ref{lem:positive-response} on cylinders of radius $R/4$ contained in $B(\xi,R)\times(0,R^2)$ gives
\begin{equation}\label{eqn:cdc-ball-lower}
 v(x,t)\ge c>0
\end{equation}
for $x\in B(\xi,R/2)$ and $R^2/2\le t\le R^2$, where $c\le1$ depends only on {$n$ and $L$}. Comparison on interior exhaustions of $U$ gives $0\le w\le u$; the comparison function $u$ is nonnegative.

Choose $\eta\in C_c^\infty(B(\xi,2R))$ with $0\le\eta\le1$ and $\eta=1$ in a neighborhood of $K$. Let $h$ be the variational solution of $\mathcal Hh=0$ in $(B(\xi,2R)\setminus K)\times(0,R^2)$, with boundary values one on $K$, zero on $\partial B(\xi,2R)$, and initial value $\eta$. More precisely, $h-\eta$ belongs to $L^2((0,R^2);H_0^1(B(\xi,2R)\setminus K))$ and has zero initial value. Extending $h$ by one on $K$, we have $h(\cdot,t)\in H_0^1(B(\xi,2R))$ for almost every $t$, and $0\le h\le1$.

The function $v-w$ solves the homogeneous equation in $U\times(0,R^2)$. Its initial value is zero, its outer boundary value is zero, and its boundary value on $\partial D\cap B(\xi,R)$ is at most one. These boundary comparisons are understood in the Sobolev sense. Since $h(\cdot,0)=\eta\ge0$, the comparison principle gives $v-w\le h$. Choose $0<\theta<1/2$ so that $C\theta^{\ass}\le c/2$. By \eqref{eqn:cdc-local-bound}, \eqref{eqn:cdc-ball-lower}, and $\dist(x,\partial D)\le|x-\xi|$, we obtain $h\ge c/2$ in $(D\cap B(\xi,\theta R))\times(R^2/2,R^2)$. On $K$, the extension equals one. Consequently,
\begin{equation}\label{eqn:cdc-energy-lower}
 \int_{R^2/2}^{R^2}\int_{B(\xi,2R)}h^2\,dx\,dt
 \ge c\theta^nR^{n+2}.
\end{equation}

Testing the time-regularized equation for $h$ with $h-\eta$, integrating in time, and using ellipticity and {Young's inequality} gives
\[
 \frac12\|h(\cdot,R^2)-\eta\|_{L^2(B(\xi,2R))}^2
 +\frac12\int_0^{R^2}\int_{B(\xi,2R)}|\nabla h|^2\,dx\,dt
 \le CL^2R^2\int_{B(\xi,2R)}|\nabla\eta|^2\,dx.
\]
There is no initial term because $h(\cdot,0)=\eta$. The Poincar\'e inequality therefore yields
\[
 \int_0^{R^2}\int_{B(\xi,2R)}h^2\,dx\,dt
 \le CR^4\int_{B(\xi,2R)}|\nabla\eta|^2\,dx.
\]
Combining this with \eqref{eqn:cdc-energy-lower} and taking the infimum over $\eta$, we conclude that
\[
 \capacity(K,B(\xi,2R))\ge cR^{n-2}.
\]
Since $\capacity(\overline{B(\xi,R)},B(\xi,2R))$ is a dimensional constant times $R^{n-2}$, this proves {\eqref{eqn:CDC}} with the stated dependence of the constants.
\end{proof}

{
\begin{remark}[Necessity of CDC for divergence-data estimates]\label{rem:cdc-divergence}
The uniform boundary estimate for divergence data in Theorem~\ref{thm:divergence} also implies {\eqref{eqn:CDC}}. More precisely, fix $A$ satisfying \eqref{eqn:A}. Suppose that, for some $0<\ass<1$ and $C>0$, {for every finite $T>0$ and every vector field satisfying \eqref{eqn:data}, there exists a bounded continuous weak solution} of $\mathcal Hu=\divergence\bm{F}$ with zero initial and lateral values such that
\[
 |u(X)|\le CM'\delta(X)^{\ass},
\]
where $\ass$ and $C$ are independent of {$T$, $\bm{F}$, and $M'$}. Then $D$ satisfies {\eqref{eqn:CDC}}, with a constant depending only on {$n$, $L$, $\ass$, and $C$}. Combined with Theorem~\ref{thm:divergence}, this gives a characterization by divergence-data estimates.

To prove necessity, take {$\xi$, $R$, and $U$} as in the proof of Theorem~\ref{thm:cdc-characterization}. Extend $\bm{1}_U$ by zero to $\R^n$, and on $D\times(0,R^2)$ set
\[
 {\bm{F}}(x,t)=\left(R^{-2}\int_{-\infty}^{x_1}
                  \bm{1}_U(a,x_2,\ldots,x_n)\,da,0,\ldots,0\right).
\]
Then $\divergence{\bm{F}}=R^{-2}\bm{1}_U$ in distributions and $|{\bm{F}}|\le2R^{-1}$. Every admissible cylinder on this time interval has $r<\delta(X)/4\le R/4$, so
\[
 \|{\bm{F}}\|_{L^2(Q(X,r))}
 \le CR^{-1}r^{n/2+1}
 \le CR^{-\ass}r^{n/2+\ass}.
\]
The assumed estimate therefore gives $|u(X)|\le C(\delta(X)/R)^{\ass}$. Since the source is nonnegative, comparison gives $u\ge0$. Thus \eqref{eqn:cdc-local-bound} holds, and the remainder of the proof of Theorem~\ref{thm:cdc-characterization} gives {\eqref{eqn:CDC}}. The vector field need not have compact support; the distribution $\divergence{\bm{F}}$ is supported at distance scale at most $R$, as required in the support formulation of Theorem~\ref{thm:divergence}.
\end{remark}
}

\section{{Measure spaces and boundary data}}\label{sec:boundary}

{
In this section and Section~\ref{sec:periodic}, assume
\begin{equation}\label{eqn:finite-global-scale}
 \|\delta\|_{L^\infty(D\times(0,T))}<\infty.
\end{equation}
This is automatic for finite $T$.
}

\begin{definition}\label{def:morrey}
{
For $0<\beta\le1$ and a locally finite signed Radon measure $\nu$, set
}
\begin{equation}\label{eqn:morrey-norm}
 {\|\nu\|_{\mathsf M^{n+\beta}(D\times(0,T))}}
 :=\|\delta\|_{L^\infty(D\times(0,T))}^{\beta}
 \sup_{\substack{X\in D\times(0,T)\\0<r<\delta(X)/4}}
 r^{-n-\beta}|\nu|(Q(X,r)).
\end{equation}
The space $\mathsf M^{n+\beta}(D\times(0,T))$ consists of measures for which this norm is finite. We also use
\begin{equation}\label{eqn:drift-norm}
 {\|\nu\|_{\mathsf M^{n+2\beta}(D\times(0,T))}}
 :=\|\delta\|_{L^\infty(D\times(0,T))}^{2\beta}
 \sup_{\substack{X\in D\times(0,T)\\0<r<\delta(X)/4}}
 r^{-n-2\beta}|\nu|(Q(X,r)),
\end{equation}
and denote the corresponding space by $\mathsf M^{n+2\beta}(D\times(0,T))$.
\end{definition}
{{When the underlying cylinder is fixed, we sometimes omit the domain from the notation for these spaces.} Any omission of a domain in a norm is specified locally.} These norms use only interior cylinders. Background on Morrey spaces and different domain restrictions can be found in \cite{Adams2015,HSS2018}.

\begin{lemma}\label{lem:complete}
Both spaces in Definition~\ref{def:morrey} are Banach spaces.
\end{lemma}
\begin{proof}
{
We prove completeness for $\mathsf M^{n+\beta}(D\times(0,T))$. Let $(\nu_j)$ be a Cauchy sequence in this space. For every compact set $K\Subset D\times(0,T)$, choose finitely many admissible cylinders $Q(X_i,r_i)$ covering $K$. By \eqref{eqn:morrey-norm},
\[
 |\nu_j-\nu_k|(K)
 \le\sum_i|\nu_j-\nu_k|(Q(X_i,r_i))
 \le\|\delta\|_{L^\infty(D\times(0,T))}^{-\beta}
 \sum_i r_i^{n+\beta}\,
 \|\nu_j-\nu_k\|_{\mathsf M^{n+\beta}(D\times(0,T))}.
\]
Hence the restrictions form a Cauchy sequence in total variation on each compact set. Completeness of the space of finite signed measures, applied on a compact exhaustion, gives compatible limits and thus a locally finite signed Radon measure $\nu$.

Fix $\epsilon>0$. Choose $j_0$ so that the norm of $\nu_j-\nu_k$ is at most $\epsilon$ for {$j\ge j_0$ and $k\ge j_0$}. For every admissible cylinder,
\[
 \|\delta\|_{L^\infty(D\times(0,T))}^{\beta}
 r^{-n-\beta}|\nu_j-\nu_k|(Q(X,r))\le\epsilon.
\]
Its closure is compactly contained in $D\times(0,T)$. Therefore total-variation convergence on that closure permits $k\to\infty$ and gives
\[
 \|\delta\|_{L^\infty(D\times(0,T))}^{\beta}
 r^{-n-\beta}|\nu_j-\nu|(Q(X,r))\le\epsilon.
\]
Taking the supremum over {$X$ and $r$} yields
\[
 \|\nu_j-\nu\|_{\mathsf M^{n+\beta}(D\times(0,T))}\le\epsilon,
\]
{for every $j\ge j_0$.} Thus $\nu$ belongs to the space and $\nu_j$ converges to $\nu$ in norm. Replacing $\beta$ by $2\beta$ in this argument proves completeness for the second space. The finite covering estimate also shows that a measure with zero norm vanishes on every compact set and hence is zero. Homogeneity and the triangle inequality follow from the corresponding properties of total variation.

}\end{proof}

For $0<\alpha\le1$ and $E\subset\overline{D\times(0,T)}$, define
\begin{equation}\label{eqn:holder-norm}
 \|u\|_{C^{\alpha,\alpha/2}(E)}
 :=\|u\|_{L^\infty(E)}+
 \|\delta\|_{L^\infty(D\times(0,T))}^{\alpha}
 [u]_{C^{\alpha,\alpha/2}(E)}.
\end{equation}
For functions on $\overline D$, define
\[
 \|u_0\|_{C^\alpha(\overline D)}
 :=\|u_0\|_{L^\infty(\overline D)}+
 \|\delta\|_{L^\infty(D\times(0,T))}^{\alpha}
 [u_0]_{C^\alpha(\overline D)},
\]
where the latter seminorm uses Euclidean distance. Thus the normalization of the initial-data norm also depends on the chosen time interval. On a fixed cylinder these are equivalent to the usual bounded H\"older norms.
\subsection{Compatible nonzero data}\label{sec:homogeneous-boundary}
Let $0<\alpha\le1$, let $g\in C^{\alpha,\alpha/2}(\partial D\times[0,T])$, and let $u_0\in C^\alpha(\overline D)$. Assume
\begin{equation}\label{eqn:boundary-compatibility}
 g(\xi,0)=u_0(\xi)
\end{equation}
for every $\xi\in\partial D$.

\begin{lemma}\label{lem:homogeneous-boundary}
Assume \eqref{eqn:A} and \eqref{eqn:CDC}. The problem
\begin{equation}\label{eqn:h}
 \begin{cases}
 \mathcal Hh=0&\text{in }D\times(0,T),\\
 h=g&\text{on }\partial D\times(0,T),\\
 h(\cdot,0)=u_0&\text{in }D
 \end{cases}
\end{equation}
has a unique bounded continuous weak solution. For a sufficiently small positive exponent $\agh$, now also depending on $\alpha$, it satisfies
\begin{equation}\label{eqn:h-holder}
 \|h\|_{C^{\agh,\agh/2}(\overline{D\times(0,T)})}
 \le C\left(\|g\|_{C^{\alpha,\alpha/2}(\partial D\times[0,T])}
       +\|u_0\|_{C^\alpha(\overline D)}\right).
\end{equation}
The exponent and constant depend only on {$n$, $L$, $\gamma$, and $\alpha$}.
\end{lemma}
\begin{proof}
{The compatibility condition allows $g$ and $u_0$ to define a continuous function on the parabolic boundary.} Its H\"older seminorm is controlled by the sum of the two seminorms: for an initial point $(x,0)$ and a lateral point $(\xi,t)$, insert the common value $g(\xi,0)=u_0(\xi)$.
The homogeneous Dirichlet construction by exhaustion, the maximum principle, and the local capacitary boundary estimate give a bounded continuous solution; this is the case $f=0$ in \cite[Theorem~6.1]{haraParabolic}. The same construction works on unbounded domains by local exhaustion and bounded comparison. The initial face is regular independently of the spatial boundary, and the compatibility condition handles their intersection.

{
Let $\widetilde g$ be the compatible function on the initial and lateral faces. For $P$ on the parabolic boundary, let $\mathcal Q_P(r)$ be the local boundary cylinder, taken forward from time zero at an initial point and backward at a lateral point. Set
\[
 \Omega_P(r)=\sup_{\mathcal Q_P(r)\cap(D\times(0,T))}|h-\widetilde g(P)|.
\]
At a lateral point away from time zero, use the capacitary decay estimate cited in Proposition~\ref{thm:std-barrier}, subtracting the constant $\widetilde g(P)$ and comparing with the oscillation of the boundary values. At an initial point whose spatial neighborhood lies in $D$, use the homogeneous initial-boundary estimate in \cite[Chapter~III]{LSU}. Near the intersection of the two faces, the same comparison uses the common boundary value supplied by \eqref{eqn:boundary-compatibility}; the lateral decay and the initial estimate apply on their respective portions. These estimates yield a common structural exponent $\kappa>0$ and the bound
\[
 \Omega_P(\theta r)
 \le C\theta^\kappa\Omega_P(r)
 +C[\widetilde g]_{C^{\alpha,\alpha/2}(\partial_p(D\times(0,T)))}r^\alpha.
\]
Choose $0<\agh<\min\{\alpha,\kappa\}$ and then $\theta$ so that $C\theta^\kappa\le\theta^\agh/2$. Iteration yields
\[
 \Omega_P(\theta^kR)
 \le C\theta^{k\agh}\left(\Omega_P(R)
 +[\widetilde g]_{C^{\alpha,\alpha/2}(\partial_p(D\times(0,T)))}R^\alpha\right).
\]
Thus the boundary oscillation is controlled with exponent $\agh$. Combining this bound with the interior homogeneous estimate, by the two-point splitting in the proof of Proposition~\ref{prop:two-point}, gives the corresponding estimate for pairs of interior points.
} Truncate the outer radius at a fixed multiple of $\|\delta\|_{L^\infty(D\times(0,T))}$. For separations below this scale the resulting estimate is
\[
 |h(X)-h(Y)|\le C\left(
 {\|g\|_{C^{\alpha,\alpha/2}(\partial D\times[0,T])}}+{\|u_0\|_{C^\alpha(\overline D)}}\right)
 \left(\frac{d_p(X,Y)}{\|\delta\|_{L^\infty(D\times(0,T))}}\right)^{\agh}.
\]
For larger separations use the maximum principle. This proves \eqref{eqn:h-holder}, including when the scale is limited by $\sqrt T$. Bounded comparison proves uniqueness.
\end{proof}

\begin{theorem}\label{thm:inhom-principal}
Let $0<\beta\le1$. Assume \eqref{eqn:A}, \eqref{eqn:CDC}, and $\nu\in\mathsf M^{n+\beta}(D\times(0,T))$. {
For compatible data {$g$ and $u_0$} satisfying the hypotheses of Lemma~\ref{lem:homogeneous-boundary},
} the problem with forcing $\nu$, lateral value $g$, and initial value $u_0$ has a unique bounded continuous weak solution. {After decreasing the exponent $\agh$ in Theorem~\ref{thm:main} to accommodate $\alpha$ if necessary,}
\begin{equation}\label{eqn:inhom-estimate}
\begin{aligned}
 \|u\|_{C^{\agh,\agh/2}(\overline{D\times(0,T)})}
 \le C\bigl(&{\|\nu\|_{\mathsf M^{n+\beta}(D\times(0,T))}}
 +\|g\|_{C^{\alpha,\alpha/2}(\partial D\times[0,T])}\\
 &+\|u_0\|_{C^\alpha(\overline D)}\bigr).
\end{aligned}
\end{equation}
\end{theorem}
\begin{proof}
{Let $v$ be {the solution with zero initial and Dirichlet boundary values supplied by} Theorem~\ref{thm:main}, and put $u=v+h$. Substituting
\[
 M=\|\delta\|_{L^\infty(D\times(0,T))}^{-\beta}
 {\|\nu\|_{\mathsf M^{n+\beta}(D\times(0,T))}}
\]
into \eqref{eqn:main-estimate} and \eqref{eqn:hoelder-estimate}, and using
$\|\delta\|_{L^\infty(\spt\nu)}\le\|\delta\|_{L^\infty(D\times(0,T))}$, gives the normalized bound for $v$. Combining this bound with \eqref{eqn:h-holder} proves \eqref{eqn:inhom-estimate}. Lowering an exponent preserves the normalized bound by using the seminorm below the global distance scale and the supremum norm above it. The comparison principle applied to the difference gives uniqueness.}
\end{proof}

\subsection{Measure and divergence data}
\begin{theorem}\label{thm:mixed}
Assume \eqref{eqn:A} and \eqref{eqn:CDC}, under the standing assumption \eqref{eqn:finite-global-scale}. Let $0<\beta<1$ and let $\nu$ be a locally finite signed Radon measure. {Suppose $\bm{F}\in L^2_{\mathrm{loc}}(D\times(0,T);\R^n)$. Assume that there are finite constants {$M$ and $M'$} such that, for every $X\in D\times(0,T)$ and $0<r<\delta(X)/4$,}
\begin{equation}\label{eqn:mixed-data}
\begin{split}
 |\nu|(Q(X,r))&\le Mr^{n+\beta},\\
 \|{\bm{F}}\|_{L^2(Q(X,r))}
 &\le M'r^{n/2+\beta}.
\end{split}
\end{equation}
{Choose {$\ass$, $\ain$, and $\agh$} as in Theorem~\ref{thm:main}.}
Then the problem
\begin{equation}\label{eqn:mixed-pde}
 \mathcal Hu=\nu+\operatorname{div}{\bm{F}}
 \quad\text{in }D\times(0,T),\qquad
 u=0\quad\text{on }\partial_p(D\times(0,T))
\end{equation}
has a unique bounded weak solution in
$V_{\mathrm{loc}}\cap C(\overline{D\times(0,T)})$. It satisfies
\begin{equation}\label{eqn:mixed-boundary}
 |u(X)|\le C(M+M')
 \|\delta\|_{L^\infty(D\times(0,T))}^{\beta-\ass}
 \delta(X)^{\ass}
\end{equation}
and
\begin{equation}\label{eqn:mixed-holder}
 [u]_{C^{\agh,\agh/2}(\overline{D\times(0,T)})}
 \le C(M+M')\|\delta\|_{L^\infty(D\times(0,T))}^{\beta-\agh}.
\end{equation}
Using the normalized H\"older norm in \eqref{eqn:holder-norm}, these bounds imply
\begin{equation}\label{eqn:mixed-norm}
 \|u\|_{C^{\agh,\agh/2}(\overline{D\times(0,T)})}
 \le C\left(
 {\|\nu\|_{\mathsf M^{n+\beta}(D\times(0,T))}}
 +{\||{\bm{F}}|^2dX\|_{\mathsf M^{n+2\beta}(D\times(0,T))}}^{1/2}\right).
\end{equation}
{Constants depend only on the structural parameters, $\beta$, and the chosen exponents, and are independent of {$\operatorname{diam}D$ and $T$}.}
\end{theorem}
A weak solution in this statement satisfies
\[
\begin{aligned}
 -\iint_{D\times(0,T)}u\partial_t\varphi\,dX
 +\iint_{D\times(0,T)}A\nabla u\cdot\nabla\varphi\,dX
 &=\int_{D\times(0,T)}\varphi\,d\nu\\
 &\quad-\iint_{D\times(0,T)}{\bm{F}}\cdot\nabla\varphi\,dX,
\end{aligned}
\]
for $\varphi\in C_c^\infty(D\times(0,T))$. The divergence term is interpreted in distributions.

\begin{proof}
By Theorem~\ref{thm:main}, the measure problem has a bounded continuous solution with zero boundary values. Its estimates imply \eqref{eqn:mixed-boundary} for the measure component, since the support scale is at most the global scale. {For the divergence component, \eqref{eqn:mixed-data} gives
\[
 \|{\bm{F}}\|_{L^2(Q(X,r))}
 \le M'\|\delta\|_{L^\infty(D\times(0,T))}^{\beta-\ass}r^{n/2+\ass},
\]
since $r<\delta(X)/4$ and $\beta>\ass$. Theorem~\ref{thm:divergence} therefore gives its contribution to \eqref{eqn:mixed-boundary}.} Adding the two solutions gives existence and \eqref{eqn:mixed-boundary}. Uniqueness follows by the comparison principle for the homogeneous difference.

{
Fix $Q(X_0,R)$ with $R\le\delta(X_0)/64$. Restriction and zero extension preserve the growth bounds by the local part of Lemma~\ref{lem:local-potential}. Approximate the positive and negative measure components by bounded densities with uniform growth constants, as in Proposition~\ref{prop:approximation}. Lemma~\ref{lem:interior-forced} applies to these approximating equations. Passing to the locally uniform limit, with local weak energy convergence as in the proof of Theorem~\ref{thm:main}, gives \eqref{eqn:interior-forced} for the mixed-data solution, with the constants {$M$ and $M'$} in \eqref{eqn:mixed-data}. The same approximation supplies its continuous interior representative.
}

{Apply the two-point argument in the proof of Proposition~\ref{prop:two-point}, using \eqref{eqn:mixed-boundary} for the boundary decay and \eqref{eqn:interior-forced} for the local oscillation. The scale conversion required there is
\[
 (M+M')R^\beta
 \le (M+M')\|\delta\|_{L^\infty(D\times(0,T))}^{\beta-\ass}R^\ass
\]
for $R\le\|\delta\|_{L^\infty(D\times(0,T))}$; the boundary estimate bounds the outer oscillation by the same expression.} Taking $\agh=\min\{\alpha_{\mathrm{in}},\ass\}$ proves \eqref{eqn:mixed-holder}. Separations larger than the global scale are handled by the supremum bound. Pass to the closure by continuity. {To pass to the normalized full-domain norm, the boundary estimate first gives
\[
 \|u\|_{L^\infty(D\times(0,T))}
 \le C(M+M')\|\delta\|_{L^\infty(D\times(0,T))}^{\beta}.
\]
Adding this to \eqref{eqn:mixed-holder} multiplied by
$\|\delta\|_{L^\infty(D\times(0,T))}^{\agh}$ gives
\[
 \|u\|_{C^{\agh,\agh/2}(\overline{D\times(0,T)})}
 \le C(M+M')\|\delta\|_{L^\infty(D\times(0,T))}^{\beta}.
\]
Finally choose {$M$ and $M'$} to be the least constants in \eqref{eqn:mixed-data}. By the definitions of the measure norms,
\[
\begin{split}
 M\|\delta\|_{L^\infty(D\times(0,T))}^{\beta}
 &=\|\nu\|_{\mathsf M^{n+\beta}(D\times(0,T))},\\
 M'\|\delta\|_{L^\infty(D\times(0,T))}^{\beta}
 &=\||{\bm{F}}|^2dX\|_{\mathsf M^{n+2\beta}(D\times(0,T))}^{1/2}.
\end{split}
\]
Substitution proves \eqref{eqn:mixed-norm}.}
\end{proof}
{\begin{corollary}\label{cor:mixed-boundary}
Under the assumptions of Theorem~\ref{thm:mixed}, let the lateral datum $g$ and initial datum $u_0$ satisfy the compatibility and H\"older assumptions of Lemma~\ref{lem:homogeneous-boundary}, with exponent $\alpha$. Then \eqref{eqn:mixed-pde} with these nonzero values has a unique bounded continuous weak solution. After decreasing $\agh$ to accommodate $\alpha$ if necessary,
\[
\begin{aligned}
 \|u\|_{C^{\agh,\agh/2}(\overline{D\times(0,T)})}
 \le C\bigl(&{\|\nu\|_{\mathsf M^{n+\beta}(D\times(0,T))}}
 +{\||{\bm{F}}|^2dX\|_{\mathsf M^{n+2\beta}(D\times(0,T))}}^{1/2}\\
 &+\|g\|_{C^{\alpha,\alpha/2}(\partial D\times[0,T])}
 +\|u_0\|_{C^\alpha(\overline D)}\bigr).
\end{aligned}
\]
All norms are normalized as in \eqref{eqn:holder-norm}.
\end{corollary}}
\begin{proof}
Add the homogeneous extension supplied by Lemma~\ref{lem:homogeneous-boundary} to the zero-boundary solution in Theorem~\ref{thm:mixed}. Combine their normalized estimates, as in Theorem~\ref{thm:inhom-principal}.
\end{proof}

{
\subsection{Gradient estimates for mixed data}\label{subsec:mixed-gradient}
The mixed-data estimate also controls the local energy at a positive Morrey exponent. This consequence will be used for the lower-order perturbations in Section~\ref{sec:lower-order}.

\begin{lemma}\label{lem:mixed-gradient}
Under the assumptions of Theorem~\ref{thm:mixed}, let $u$ be {the solution with zero initial and Dirichlet boundary values}. For every $X\in D\times(0,T)$ and $0<r<\delta(X)/4$,
\begin{equation}\label{eqn:mixed-gradient-growth}
 \iint_{Q(X,r)}|\nabla u|^2\,dY
 \le C(M+M')^2
 \|\delta\|_{L^\infty(D\times(0,T))}^{2(\beta-\agh)}
 r^{n+2\agh}.
\end{equation}
Equivalently,
\begin{equation}\label{eqn:mixed-gradient-norm}
 \||\nabla u|^2dX\|_{\mathsf M^{n+2\agh}(D\times(0,T))}^{1/2}
 \le C\left(\|\nu\|_{\mathsf M^{n+\beta}(D\times(0,T))}
 +\||{\bm{F}}|^2dX\|_{\mathsf M^{n+2\beta}(D\times(0,T))}^{1/2}\right).
\end{equation}
The norm on the left is defined by \eqref{eqn:drift-norm} with $\beta$ replaced by $\agh$. Constants have the same dependencies as in Theorem~\ref{thm:mixed}.
\end{lemma}
\begin{proof}
First suppose $r<\delta(X)/16$, so that the data growth assumptions apply on $Q(X,2r)$. Put $w=u-u(X)$ and choose spatial and time cutoffs supported in that cylinder, equal to one on $Q(X,r)$, with bounds $C/r$ and $C/r^2$ on their respective derivatives. Testing the equation with $w$ times the squared cutoff, absorbing the cross terms by ellipticity, and discarding the nonnegative terminal term gives
\begin{equation}\label{eqn:mixed-caccioppoli}
\begin{split}
 \iint_{Q(X,r)}|\nabla u|^2\,dY
 \le{}&Cr^{-2}\iint_{Q(X,2r)}|w|^2\,dY\\
 &+C\int_{Q(X,2r)}|w|\,d|\nu|
 +C\iint_{Q(X,2r)}|{\bm{F}}|^2\,dY.
\end{split}
\end{equation}
For the divergence term, the derivative falling on the cutoff is bounded by the first and third terms, using {Young's inequality}. The calculation is justified by local time regularization and approximation of the measure and divergence data; uniform convergence in the continuous measure pairing and weak lower semicontinuity of the energy give the displayed inequality for the solution of Theorem~\ref{thm:mixed}.

The H\"older estimate of that theorem bounds the oscillation of $u$ on $Q(X,2r)$ by
\[
 C(M+M')\|\delta\|_{L^\infty(D\times(0,T))}^{\beta-\agh}r^{\agh}.
\]
Consequently, the three terms in \eqref{eqn:mixed-caccioppoli} are bounded, respectively, by
\[
\begin{gathered}
 C(M+M')^2\|\delta\|_{L^\infty(D\times(0,T))}^{2(\beta-\agh)}r^{n+2\agh},\\
 CM(M+M')\|\delta\|_{L^\infty(D\times(0,T))}^{\beta-\agh}r^{n+\beta+\agh},
 \qquad C(M')^2r^{n+2\beta}.
\end{gathered}
\]
Since $\agh<\beta$ and $r<\|\delta\|_{L^\infty(D\times(0,T))}$, this proves \eqref{eqn:mixed-gradient-growth} in the smaller range.

For a general admissible radius, cover the spatial ball by finitely many balls of radius $r/8$ with centers in the original ball, and partition the time interval into $65$ equal intervals. The backward cylinders of radius $r/8$ at the corresponding upper endpoints cover $Q(X,r)$. At their centers $Y$, $\delta(Y)>3r$, so their radii satisfy the smaller-range condition. Summation proves the full assertion. Express the least constants {$M$ and $M'$} by \eqref{eqn:morrey-norm}--\eqref{eqn:drift-norm} to obtain \eqref{eqn:mixed-gradient-norm}.
\end{proof}

}

{
\section{Equations with lower-order terms}\label{sec:lower-order}
{
The elliptic problem with locally integrable lower-order coefficients was treated in \cite{HaraEllipticLower} by a Morrey-space estimate and the Fredholm alternative. Here the principal solution operator acts on a sum of measure and divergence sources. We estimate the perturbation on short time intervals, invert it by a Neumann series, and continue the solution in time.
}

\subsection{The equation and growth conditions}
Let $D\subset\R^n$, $n\ge2$, satisfy {\eqref{eqn:CDC}}, and let $A$ satisfy the boundedness and uniform ellipticity condition \eqref{eqn:A}. {On each finite strip, $\delta\le\sqrt T$, so the normalized spaces and estimates of Section~\ref{sec:boundary} apply.} Fix $0<\beta<1$. We consider
\begin{equation}\label{eqn:pde}
 \mathcal Hu+{\bm{b}}\cdot\nabla u+\operatorname{div}({\bm{c}} u)+u\mu
 =\nu+\operatorname{div}{\bm{F}}
 \quad\text{in }D\times(0,\infty).
\end{equation}
The fields {${\bm{b}}$, ${\bm{c}}$, and ${\bm{F}}$} are measurable; {$\mu$ and $\nu$} are locally finite signed Radon measures. All divergence operators act on the spatial variables.
For every finite $T>0$, assume the following estimates on all admissible cylinders, with finite growth constants depending on $T$:
\begin{equation}\label{eqn:growth}
\begin{aligned}
 \iint_{Q(X,r)}|{\bm{b}}|^2\,dY&\le M_{{\bm{b}},T}^2r^{n+2\beta},&
 \iint_{Q(X,r)}|{\bm{c}}|^2\,dY&\le M_{{\bm{c}},T}^2r^{n+2\beta},\\
 |\mu|(Q(X,r))&\le M_{\mu,T}r^{n+\beta},&
 |\nu|(Q(X,r))&\le M_T r^{n+\beta},\\
 \iint_{Q(X,r)}|{\bm{F}}|^2\,dY&\le (M'_T)^2r^{n+2\beta}.
\end{aligned}
\end{equation}
These growth conditions are imposed on interior cylinders, and the drift divergences in \eqref{eqn:pde} are understood in distributions.

Prescribe $u=g$ on $\partial D\times[0,\infty)$ and $u(\cdot,0)=u_0$. Assume that $u_0$ is bounded and H\"older continuous on $\overline D$, that $g$ is bounded and parabolically H\"older continuous on each finite lateral strip, and that $g(\cdot,0)=u_0$ on $\partial D$. Fix a common data exponent $0<\alpha\le1$. For $u\in V_{\loc}\cap C(\overline{D\times(0,T)})$, the weak formulation is
\begin{equation}\label{eqn:lower-weak-formulation}
\begin{aligned}
 -\iint_{D\times(0,T)}u\partial_t\varphi\,dX
 +\iint_{D\times(0,T)}A\nabla u\cdot\nabla\varphi\,dX
 &+\iint_{D\times(0,T)}({\bm{b}}\cdot\nabla u)\varphi\,dX\\
 -\iint_{D\times(0,T)}u{\bm{c}}\cdot\nabla\varphi\,dX
 +\int_{D\times(0,T)}u\varphi\,d\mu
 &=\int_{D\times(0,T)}\varphi\,d\nu
 -\iint_{D\times(0,T)}{\bm{F}}\cdot\nabla\varphi\,dX
\end{aligned}
\end{equation}
for every $\varphi\in C_c^\infty(D\times(0,T))$.

\subsection{The mixed source space}\label{subsec:mixed-source-space}
Fix a finite $T>0$. Define
\[
\begin{aligned}
 \mathsf V^{2,n+2\beta}(D\times(0,T))
 &=\{{\bm{F}}:|{\bm{F}}|^2dX\in\mathsf M^{n+2\beta}(D\times(0,T))\},\\
 \|{\bm{F}}\|_{\mathsf V^{2,n+2\beta}(D\times(0,T))}
 &=\||{\bm{F}}|^2dX\|_{\mathsf M^{n+2\beta}(D\times(0,T))}^{1/2}.
\end{aligned}
\]
This vector-field space is distinct from the local energy class $V_{\mathrm{loc}}$. Define the space of distributions
\begin{equation}\label{eqn:source}
\begin{aligned}
 \mathsf X^{n+\beta}(D\times(0,T))
 &=\mathsf M^{n+\beta}(D\times(0,T))
   +\operatorname{div}\mathsf V^{2,n+2\beta}(D\times(0,T)),\\
 \|\sigma\|_{\mathsf X^{n+\beta}(D\times(0,T))}
 &=\inf_{\sigma=\nu+\operatorname{div}{\bm{F}}}
 \bigl(\|\nu\|_{\mathsf M^{n+\beta}(D\times(0,T))}\\
 &\hspace{30mm}+\|{\bm{F}}\|_{\mathsf V^{2,n+2\beta}(D\times(0,T))}\bigr).
\end{aligned}
\end{equation}
{In space names and norms in this subsection, an omitted domain is $D\times(0,T)$.}
The equality in the infimum holds in distributions on $D\times(0,T)$.

\begin{proposition}\label{prop:mixed-source}
The spaces $\mathsf V^{2,n+2\beta}$ and $\mathsf X^{n+\beta}$ are Banach spaces. {The solution operator for zero initial and Dirichlet boundary values} ${\bm{\mathrm{G}}_0}$ is well defined, linear, and injective on $\mathsf X^{n+\beta}$, and
\begin{equation}\label{eqn:mixed-G0}
\begin{split}
 \mathcal H{\bm{\mathrm{G}}_0}\sigma&=\sigma,\\
 \|{\bm{\mathrm{G}}_0}\sigma\|_{C^{\agh,\agh/2}(\overline{D\times(0,T)})}
 +\||\nabla{\bm{\mathrm{G}}_0}\sigma|^2dX\|_{\mathsf M^{n+2\agh}(D\times(0,T))}^{1/2}
 &\le C\|\sigma\|_{\mathsf X^{n+\beta}(D\times(0,T))}.
\end{split}
\end{equation}
\end{proposition}
\begin{proof}
{
Let $(\bm F_j)$ be Cauchy in $\mathsf V^{2,n+2\beta}$. As in the proof of Lemma~\ref{lem:complete}, finite admissible coverings give a local limit, now by completeness of $L^2$ on compact subsets. Passing to this limit on each admissible cylinder and then taking the supremum gives convergence in $\mathsf V^{2,n+2\beta}$. Thus this space is complete. Lemma~\ref{lem:complete} then shows that its product with $\mathsf M^{n+\beta}$, equipped with the sum norm, is a Banach space.
}

The map $(\nu,{\bm{F}})\mapsto\nu+\operatorname{div}{\bm{F}}$ into distributions is continuous: for every compactly supported smooth $\varphi$,
\[
 \left|\int_{D\times(0,T)}\varphi\,d\nu
 -\iint_{D\times(0,T)}{\bm{F}}\cdot\nabla\varphi\,dX\right|
 \le C_\varphi\left(\|\nu\|_{\mathsf M^{n+\beta}(D\times(0,T))}+\|{\bm{F}}\|_{\mathsf V^{2,n+2\beta}(D\times(0,T))}\right).
\]
Its kernel is closed, and the corresponding Banach quotient is precisely \eqref{eqn:source}. {For any representation of $\sigma$, Theorem~\ref{thm:mixed} supplies the solution corresponding to $\sigma$.} Different representations give the same solution by homogeneous comparison. Take the infimum over representations in that theorem and Lemma~\ref{lem:mixed-gradient} to obtain \eqref{eqn:mixed-G0}. Uniqueness proves linearity; the equation proves injectivity.
\end{proof}

\subsection{Uniform smallness on short time intervals}\label{subsec:short-intervals}
For $0\le s<s+\tau\le T$, define
\[
 \delta_s(x,t)=\min\{\operatorname{dist}(x,\partial D),\sqrt{t-s}\}
 \quad (s<t<s+\tau).
\]
Define $\mathsf M^{n+\beta}(D\times(s,s+\tau))$ and
$\mathsf M^{n+2\beta}(D\times(s,s+\tau))$ by replacing $\delta$ with $\delta_s$ and $(0,T)$ with $(s,s+\tau)$ in \eqref{eqn:morrey-norm}--\eqref{eqn:drift-norm}. Define
$\mathsf V^{2,n+2\beta}(D\times(s,s+\tau))$ and
$\mathsf X^{n+\beta}(D\times(s,s+\tau))$ by the same vector-field and quotient constructions as \eqref{eqn:source}. {The H\"older norms on this shifted interval are normalized by $\|\delta_s\|_{L^\infty(D\times(s,s+\tau))}$; spatial initial traces use the same scale.}

{The corresponding principal solution operator for zero initial and Dirichlet boundary values} is denoted by ${\bm{\mathrm{G}}_0}$; it is applied after translation of $s$ to time zero, with terminal time $\tau$.

\begin{lemma}\label{lem:small}
The coefficient norms on every such interval satisfy
\begin{equation}\label{eqn:smallnorm}
 \|{\bm{b}}\|_{\mathsf V^{2,n+2\beta}(D\times(s,s+\tau))}
 +\|{\bm{c}}\|_{\mathsf V^{2,n+2\beta}(D\times(s,s+\tau))}
 +\|\mu\|_{\mathsf M^{n+\beta}(D\times(s,s+\tau))}
 \le \tau^{\beta/2}
 (M_{{\bm{b}},T}+M_{{\bm{c}},T}+M_{\mu,T}).
\end{equation}
The principal estimate and its gradient consequence have constants independent of {$s$ and $\tau$} and $\operatorname{diam}D$.
\end{lemma}
\begin{proof}
Since $\delta_s\le\delta$, every cylinder admissible for the shifted problem is admissible for \eqref{eqn:growth}. Also
$\|\delta_s\|_{L^\infty(D\times(s,s+\tau))}\le\sqrt\tau$.
For example,
\[
 \|{\bm{b}}\|_{\mathsf V^{2,n+2\beta}(D\times(s,s+\tau))}
 =\|\delta_s\|_{L^\infty(D\times(s,s+\tau))}^{\beta}
 \sup_{\substack{X\in D\times(s,s+\tau)\\0<r<\delta_s(X)/4}}
 r^{-n/2-\beta}
 \left(\iint_{Q(X,r)}|{\bm{b}}|^2\,dY\right)^{1/2}
 \le\tau^{\beta/2}M_{{\bm{b}},T}.
\]
The other bounds follow identically. The shifted principal problem has finite global distance scale, even when $D$ is unbounded. The structural constants in Theorem~\ref{thm:mixed} and Lemma~\ref{lem:mixed-gradient} are unchanged by time translation and parabolic rescaling.
\end{proof}

{The operator norms below are those of maps from $\mathsf X^{n+\beta}(D\times(s,s+\tau))$ into itself.}

Define on the shifted source space
\begin{equation}\label{eqn:operator}
 \mathcal T_s\sigma=-({\bm{b}}\cdot\nabla{\bm{\mathrm{G}}_0}\sigma)dX
 -\operatorname{div}({\bm{c}}\,{\bm{\mathrm{G}}_0}\sigma)
 -({\bm{\mathrm{G}}_0}\sigma)\mu.
\end{equation}
{
Put $u=\bm{\mathrm{G}}_0\sigma$. Apply the gradient term of \eqref{eqn:mixed-G0} after time translation. Since $r<\delta_s(X)/4\le\|\delta_s\|_{L^\infty(D\times(s,s+\tau))}$, its normalized Morrey bound implies
\[
 \|\nabla u\|_{L^2(Q(X,r))}
 \le C\|\sigma\|_{\mathsf X^{n+\beta}(D\times(s,s+\tau))}r^{n/2}
\]
on shifted admissible cylinders. 

The Cauchy--Schwarz inequality now gives
\[
 \int_{Q(X,r)}|\bm{b}\cdot\nabla u|\,dY
 \le CM_{\bm{b},T}\|\sigma\|_{\mathsf X^{n+\beta}(D\times(s,s+\tau))}r^{n+\beta}.
\]
Multiplying by the normalization factor yields
\[
 \|(\bm{b}\cdot\nabla u)dX\|_{\mathsf M^{n+\beta}(D\times(s,s+\tau))}
 \le C\tau^{\beta/2}M_{\bm{b},T}
 \|\sigma\|_{\mathsf X^{n+\beta}(D\times(s,s+\tau))}.
\]
For the other two terms, multiplication by a bounded function gives
\[
\begin{aligned}
 \|\bm{c}u\|_{\mathsf V^{2,n+2\beta}(D\times(s,s+\tau))}
 &\le\|u\|_{L^\infty(D\times(s,s+\tau))}
 \|\bm{c}\|_{\mathsf V^{2,n+2\beta}(D\times(s,s+\tau))},\\
 \|u\mu\|_{\mathsf M^{n+\beta}(D\times(s,s+\tau))}
 &\le\|u\|_{L^\infty(D\times(s,s+\tau))}
 \|\mu\|_{\mathsf M^{n+\beta}(D\times(s,s+\tau))}.
\end{aligned}
\]
Therefore \eqref{eqn:mixed-G0} and Lemma~\ref{lem:small}, followed by the definition of the quotient norm, give
}
\begin{equation}\label{eqn:opnorm}
 \|\mathcal T_s\|_{\mathsf X^{n+\beta}(D\times(s,s+\tau))}\le C\tau^{\beta/2}
 (M_{{\bm{b}},T}+M_{{\bm{c}},T}+M_{\mu,T}).
\end{equation}
Choose a length $\tau_T>0$ so that the right-hand side is at most $1/2$ for every $\tau\le\tau_T$. If all three constants vanish, any length is allowed. Then
\begin{equation}\label{eqn:inverse}
 (I-\mathcal T_s)^{-1}=\sum_{j=0}^{\infty}\mathcal T_s^j,
 \qquad \|(I-\mathcal T_s)^{-1}\|_{\mathsf X^{n+\beta}(D\times(s,s+\tau))}\le2.
\end{equation}

\subsection{Initial values and the solution class}
\begin{proposition}\label{prop:local}
{
Fix $0<\tau\le\tau_T$ and $0\le s<s+\tau\le T$. Choose $\agh>0$ below $\alpha$, $\beta$, and the structural interior and boundary decay exponents. For compatible bounded data $u_s\in C^\agh(\overline D)$ and $g\in C^{\agh,\agh/2}(\partial D\times[s,s+\tau])$, equation \eqref{eqn:pde} has a unique bounded continuous weak solution in $V_{\loc}(D\times(s,s+\tau))$, with initial value $u_s$ at time $s$ and lateral value $g$. It satisfies
\begin{equation}\label{eqn:short-time-holder}
\begin{split}
 \|u\|_{C^{\agh,\agh/2}(\overline D\times[s,s+\tau])}
 \le C\bigl(&\|u_s\|_{C^\agh(\overline D)}
 +\|g\|_{C^{\agh,\agh/2}(\partial D\times[s,s+\tau])}\\
 &+\|\nu+\divergence\bm{F}\|_{\mathsf X^{n+\beta}(D\times(s,s+\tau))}\bigr).
\end{split}
\end{equation}
All three H\"older norms use the scale $\|\delta_s\|_{L^\infty(D\times(s,s+\tau))}$, including the spatial initial norm. The constant depends only on {$n$, $L$, $\gamma$, $\beta$, and the fixed $\agh$}; the coefficient norms enter only through the choice of $\tau_T$ in \eqref{eqn:opnorm}. The same exponent is retained on successive intervals.
}
\end{proposition}
\begin{proof}
Let $h$ be the homogeneous principal extension of the initial and lateral values, supplied by Lemma~\ref{lem:homogeneous-boundary} after translation. {
Choose $\agh$ below the fixed homogeneous interior and boundary exponents and the original data exponent. In the boundary iteration from Lemma~\ref{lem:homogeneous-boundary}, data of exponent $\agh$ give
\[
 \Omega_P(\theta r)\le\tfrac12\theta^\agh\Omega_P(r)
 +C[\widetilde g]_{C^{\agh,\agh/2}(\partial_p(D\times(s,s+\tau)))}r^\agh.
\]
After division by $(\theta r)^\agh$, the homogeneous factor is $1/2$. Summing its geometric series retains the exponent $\agh$, rather than decreasing it at each interface. The interior estimate and the two-point argument then give
\[
 \|h\|_{C^{\agh,\agh/2}(\overline D\times[s,s+\tau])}
 \le C\left(\|u_s\|_{C^\agh(\overline D)}
 +\|g\|_{C^{\agh,\agh/2}(\partial D\times[s,s+\tau])}\right),
\]
with all norms normalized at the shifted interval scale.
} Its local Caccioppoli inequality gives
\[
 \iint_{Q(X,r)}|\nabla h|^2\,dY\le C\|h\|_{L^\infty(D\times(s,s+\tau))}^2r^n
\]
on every shifted admissible cylinder, first on smaller cylinders and then by a finite covering. Thus
\[
 \|({\bm{b}}\cdot\nabla h)dX\|_{\mathsf M^{n+\beta}(D\times(s,s+\tau))}
 \le C\|{\bm{b}}\|_{\mathsf V^{2,n+2\beta}(D\times(s,s+\tau))}\|h\|_{L^\infty(D\times(s,s+\tau))}.
\]
{The terms $\operatorname{div}(\bm{c}h)$ and $h\mu$ satisfy the corresponding product bounds.} Consequently the corrected source
\[
 \nu+\operatorname{div}{\bm{F}}-({\bm{b}}\cdot\nabla h)dX
 -\operatorname{div}({\bm{c}} h)-h\mu
\]
belongs to $\mathsf X^{n+\beta}(D\times(s,s+\tau))$. Apply \eqref{eqn:inverse} to it, then apply ${\bm{\mathrm{G}}_0}$ and add $h$. The principal estimate, the extension estimate, and the bound $2$ in \eqref{eqn:inverse} {
give \eqref{eqn:short-time-holder}, since the sum of the shifted coefficient norms is bounded by a structural constant through \eqref{eqn:opnorm}.
}

For uniqueness in the full stated class, it remains to check representation of an arbitrary bounded continuous weak solution with zero parabolic values. A cutoff test with $u$ gives
\begin{equation}\label{eqn:energy}
\begin{split}
 \iint_{Q(X,r)}|\nabla u|^2\,dY\le{}&C\|u\|_{L^\infty(D\times(s,s+\tau))}^2r^n
 +C\|u\|_{L^\infty(D\times(s,s+\tau))}|\nu|(Q(X,2r))\\
 &+C\|u\|_{L^\infty(D\times(s,s+\tau))}^2|\mu|(Q(X,2r))
 +C\iint_{Q(X,2r)}|{\bm{F}}|^2\,dY\\
 &+C\|u\|_{L^\infty(D\times(s,s+\tau))}^2
 \iint_{Q(X,2r)}(|{\bm{b}}|^2+|{\bm{c}}|^2)\,dY
\end{split}
\end{equation}
for $r<\delta_s(X)/16$. The inner drift is absorbed by ellipticity and {Young's inequality}; the divergence field is ${\bm{F}}-{\bm{c}} u$. Local regularization justifies the continuous measure pairing. At the normalized interval scale the right-hand side is at most $C_u r^n$. A finite covering extends this estimate to all admissible radii. The Cauchy--Schwarz inequality now gives
$({\bm{b}}\cdot\nabla u)dX\in\mathsf M^{n+\beta}(D\times(s,s+\tau))$.
Hence $\mathcal Hu\in\mathsf X^{n+\beta}(D\times(s,s+\tau))$, and principal uniqueness gives $u={\bm{\mathrm{G}}_0}(\mathcal Hu)$. Applying this to a homogeneous difference identifies it with a kernel element of $I-\mathcal T_s$, which vanishes by \eqref{eqn:inverse}.
\end{proof}

\subsection{Continuation and finite-time estimates}
\begin{theorem}\label{thm:continuation}
Under the hypotheses of Section~\ref{sec:lower-order}, equation \eqref{eqn:pde} has a unique weak solution which belongs to $V_{\mathrm{loc}}$, is continuous on $\overline D\times[0,\infty)$, and is bounded on every finite strip. For a fixed sufficiently small $\agh>0$ as in Proposition~\ref{prop:local}, it is H\"older continuous on every finite strip, with
\begin{equation}\label{eqn:finite-estimate}
\begin{split}
 &\|u\|_{C^{\agh,\agh/2}(\overline D\times[0,T])}\\
 &\quad\le C_T\bigl(M_T+M'_T
 +\|u_0\|_{C^\alpha(\overline D)}
 +\|g\|_{C^{\alpha,\alpha/2}(\partial D\times[0,T])}\bigr).
\end{split}
\end{equation}
Here all H\"older norms are normalized as in \eqref{eqn:holder-norm}. The constant may depend on $T$, the coefficient growth constants on that strip, and the structural parameters.
\end{theorem}
\begin{proof}
Fix $T<\infty$. Choose an integer $N$ such that $T/N\le\tau_T$ and partition $[0,T]$ into $N$ equal intervals. Solve on the first interval by Proposition~\ref{prop:local}. Its terminal trace is bounded and $C^\agh$ in space and agrees with $g$ on the spatial boundary. Use this trace as the next initial value and repeat. The same positive exponent $\agh$ is used at every step.

{
Fix an interface time $t_j$ and a compact spatial set $K\Subset D$. For $r$ sufficiently small, cover $K$ by at most $C_Kr^{-n}$ balls of radius $r$. The backward cylinders centered at time $t_j+r^2/2$ contain the slice at $t_j$ and are admissible. Hence \eqref{eqn:growth} gives
\[
 (|\mu|+|\nu|)(K\times\{t_j\})
 \le C_K(M_{\mu,T}+M_T)r^{-n}r^{n+\beta}
 \longrightarrow0.
\]
Thus restricting the measures to the open subintervals loses no mass.

Choose $K\Subset U\Subset D$ and a spatial cutoff supported in $U$ and equal to one on $K$. On the interval to the right of $t_j$, the energy test retains the initial trace and gives, for sufficiently small $a>0$,
\[
\begin{split}
 \int_{t_j}^{t_j+a}\int_K|\nabla u|^2\,dx\,dt
 \le C\biggl(&\|u(\cdot,t_j)\|_{L^2(U)}^2
 +\iint_{U\times(t_j,t_j+a)}|u|^2\,dX\\
 &+\|u\|_{L^\infty(U\times(t_j,t_j+a))}|\nu|(U\times(t_j,t_j+a))\\
 &+\|u\|_{L^\infty(U\times(t_j,t_j+a))}^2|\mu|(U\times(t_j,t_j+a))\\
 &+\|u\|_{L^\infty(U\times(t_j,t_j+a))}^2
 \iint_{U\times(t_j,t_j+a)}(|\bm{b}|^2+|\bm{c}|^2)\,dX\\
 &+\iint_{U\times(t_j,t_j+a)}|\bm{F}|^2\,dX\biggr)<\infty.
\end{split}
\]
Here $C$ also depends on the spatial cutoff. The drift and field terms are absorbed by ellipticity and {Young's inequality} as in \eqref{eqn:energy}; approximation of the initial trace and lower semicontinuity justify the test. The corresponding left-hand energy bound follows by retaining the terminal term and taking limits from below. Therefore the concatenation belongs to $V_{\loc}$ across $t_j$.

Finally, integrate the weak equation on the two adjacent intervals, using time cutoffs approaching $t_j$. Their interface contributions sum to
\[
 \int_D\bigl(u(x,t_j^-)-u(x,t_j^+)\bigr)\varphi(x,t_j)\,dx=0,
\]
since the terminal and initial traces agree continuously. The local energy bounds justify the cutoff limits, and the absence of time-slice atoms permits addition of the measure terms. Hence the concatenated function satisfies the weak formulation on the whole strip.
}

{
Iterating \eqref{eqn:short-time-holder}
} gives a finite bound depending on $N$ and the displayed parameters. Adjacent-interval H\"older estimates combine by inserting the common interface value at a fixed spatial point. Iterating this observation over the finite partition proves \eqref{eqn:finite-estimate}.

The same interval argument proves uniqueness among solutions bounded on each finite strip. Solutions constructed for different terminal times agree on their common strip by this uniqueness. Their union defines the solution on $[0,\infty)$ and retains every finite-time estimate.
\end{proof}

\subsection{Quantitative control under continuation}
All H\"older norms in the following proposition are normalized by
$\|\delta\|_{L^\infty(D\times(0,T))}$ as in \eqref{eqn:holder-norm}, including the initial and lateral data norms.
\begin{proposition}\label{prop:quantitative}
The solution in Theorem~\ref{thm:continuation} satisfies
\begin{equation}\label{eqn:exponential}
\begin{split}
 \|u\|_{C^{\agh,\agh/2}(\overline D\times[0,T])}
 \le{}& C\exp\bigl(CT (M_{{\bm{b}},T}+M_{{\bm{c}},T}+M_{\mu,T})^{2/\beta}\bigr)\\
 &\times\bigl(\|u_0\|_{C^\alpha(\overline D)}
 +\|g\|_{C^{\alpha,\alpha/2}(\partial D\times[0,T])}\\
 &\hspace{25mm}+T^{\beta/2}(M_T+M'_T)\bigr).
\end{split}
\end{equation}
The constant $C$ depends only on the structural parameters and the fixed exponents.
\end{proposition}
\begin{proof}
Choose $N\ge1$ with $\tau=T/N$ so that the bound in \eqref{eqn:opnorm} is at most $1/2$. One can arrange
\[
 N\le1+C T (M_{{\bm{b}},T}+M_{{\bm{c}},T}+M_{\mu,T})^{2/\beta}.
\]
Write $u_j$ for the restriction of $u$ to $D\times((j-1)\tau,j\tau)$. {Their H\"older norms use the shifted-interval convention of Subsection~\ref{subsec:short-intervals}. Since the subintervals have equal length $\tau$, their normalizing scales all equal $\|\delta_0\|_{L^\infty(D\times(0,\tau))}$.}
The spatial norm $\|u_0\|_{C^\agh(\overline D)}$ used for the first step is defined on $\overline D$ with this same scale. {The $C^\alpha$ norm of $u_0$ and the $C^{\alpha,\alpha/2}$ norm of $g$ on the right-hand side of the proposition use the full-interval scale $\|\delta\|_{L^\infty(D\times(0,T))}$.} For $2\le j\le N$, the endpoint trace from interval $j-1$ has spatial norm at most $\|u_{j-1}\|_{C^{\agh,\agh/2}(\overline D\times[(j-2)\tau,(j-1)\tau])}$ at the same scale. {Proposition~\ref{prop:local}, applied to the restrictions constructed in Theorem~\ref{thm:continuation}, gives,} for a fixed $C_0\ge2$ and $2\le j\le N$,
\[
\begin{aligned}
 \|u_j\|_{C^{\agh,\agh/2}(\overline D\times[(j-1)\tau,j\tau])}
 \le{}&C_0\|u_{j-1}\|_{C^{\agh,\agh/2}(\overline D\times[(j-2)\tau,(j-1)\tau])}\\
 &+C_0\bigl(\|g\|_{C^{\alpha,\alpha/2}(\partial D\times[0,T])}
 +\tau^{\beta/2}(M_T+M'_T)\bigr).
\end{aligned}
\]
Here the norm of $g$ uses the full interval scale; it controls its restrictions at the shorter scale and the smaller exponent. On the first interval the same estimate holds with the preceding-interval norm replaced by the spatial norm $\|u_0\|_{C^\agh(\overline D)}$ at the short scale, which is controlled by the full-scale $C^\alpha$ data norm. Summing the geometric progression gives
\[
\begin{split}
 \max_{1\le j\le N}\|u_j\|_{C^{\agh,\agh/2}(\overline D\times[(j-1)\tau,j\tau])}\le C C_0^N\bigl(&\|u_0\|_{C^\alpha(\overline D)}
 +\|g\|_{C^{\alpha,\alpha/2}(\partial D\times[0,T])}\\
 &+T^{\beta/2}(M_T+M'_T)\bigr).
\end{split}
\]
To control the global H\"older seminorm, pairs with time separation less than $\tau$ lie in the same or adjacent intervals. In the latter case insert the common interface value, first at one fixed spatial point. If the time separation is at least $\tau$, use the supremum norm. These estimates give the global seminorm with normalization at the short interval scale. The ratio of the full and short scales satisfies
\[
 \frac{\|\delta\|_{L^\infty(D\times(0,T))}}
 {\|\delta_0\|_{L^\infty(D\times(0,\tau))}}\le\sqrt N.
\]
Thus changing normalization contributes at most $C N^{\agh/2}$. Finally,
$N^{\agh/2}C_0^N\le C\exp(CT (M_{{\bm{b}},T}+M_{{\bm{c}},T}+M_{\mu,T})^{2/\beta})$, which proves \eqref{eqn:exponential}.
\end{proof}

\subsection{{Sharpness of the continuation estimate}}
Continuation does not provide a norm bound uniform in $T$. This remains false even for smooth coefficients and zero forcing. On a bounded smooth domain, let $\phi_1$ be a first Dirichlet eigenfunction of $-\Delta$, with eigenvalue $\lambda_1$, and take
\[
 {\bm{b}}={\bm{c}}=0,\quad \mu=-a\,dX,\quad \nu=0,\quad{\bm{F}}=0,
 \qquad a>\lambda_1.
\]
The zero lateral value solution with initial value $\phi_1$ is
\[
 u(x,t)=e^{(a-\lambda_1)t}\phi_1(x).
\]
The coefficient satisfies \eqref{eqn:growth}, even with constants uniform in $T$ because the spatial distance scale is bounded. The solution exists for all time but its supremum norm grows exponentially. A uniform estimate on $D\times(0,\infty)$ therefore needs an additional stability hypothesis.

The power $2/\beta$ in the coefficient dependence is also sharp for the full coefficient class. To see this on a fixed smooth bounded domain containing the origin, choose a nonnegative $f\in C_c^\infty(B(0,1))$ so large that some $\varphi\in C_c^\infty(B(0,1))$ satisfies
\[
 \int_{B(0,1)}f\varphi^2\,dx>
 \int_{B(0,1)}|\nabla\varphi|^2\,dx.
\]
For sufficiently small $\epsilon>0$, set
\[
 \mu_\epsilon=-\epsilon^{-2}f(x/\epsilon)\,dX,
 \qquad {\bm{b}}={\bm{c}}=0.
\]
At radii $r\le\epsilon$ and $r\ge\epsilon$, respectively,
\[
 |\mu_\epsilon|(Q(X,r))
 \le C\epsilon^{-2}r^{n+2},\qquad
 |\mu_\epsilon|(Q(X,r))
 \le C\epsilon^{n-2}r^2.
\]
Since $n\ge2$, both are bounded by
$C\epsilon^{-\beta}r^{n+\beta}$. Testing at an interior cylinder of radius comparable to $\epsilon$ also gives the reverse bound for the least growth constant, on any fixed interval with $T>0$ once $\epsilon$ is sufficiently small. Hence
$M_{\mu_\epsilon,T}\asymp\epsilon^{-\beta}$.

The Rayleigh quotient of $\varphi(x/\epsilon)$ for
$-\Delta-\epsilon^{-2}f(x/\epsilon)$ is at most
$-c\epsilon^{-2}$. Its first Dirichlet eigenvalue is therefore
$-\lambda_\epsilon$ with $\lambda_\epsilon\ge c\epsilon^{-2}$.
The corresponding initial eigenfunction produces a solution growing as
$e^{\lambda_\epsilon t}$. Normalize the eigenfunction to have supremum norm one. {The Rayleigh quotient also gives $\lambda_\epsilon\le\epsilon^{-2}\|f\|_{L^\infty(\R^n)}$. If $\phi_\epsilon$ is the normalized eigenfunction, then
\[
 -\Delta\phi_\epsilon
 =\bigl(\epsilon^{-2}f(x/\epsilon)-\lambda_\epsilon\bigr)\phi_\epsilon.
\]
Fix $q>n$. The global Dirichlet $W^{2,q}$ estimate on the fixed smooth domain gives $\|\phi_\epsilon\|_{W^{2,q}(D)}\le C(1+\epsilon^{-2})$. Sobolev embedding therefore bounds its $C^\alpha$ norm by the same polynomial for $0<\alpha\le1$.} This polynomial initial-data factor cannot absorb $e^{cT\epsilon^{-2}}$ for fixed $T>0$. Thus a uniform estimate of the form \eqref{eqn:exponential} cannot replace $(M_{{\bm{b}},T}+M_{{\bm{c}},T}+M_{\mu,T})^{2/\beta}$ in the exponential by a smaller power. Together with the preceding constant-potential example, this proves sharpness of the exponential time dependence and of the coefficient power.

}

\section{Periodic homogenization}\label{sec:periodic}
\begingroup
Periodic homogenization of parabolic divergence-form equations is classical; see \cite[Chapter~2]{BLP1978} for periodic cell problems and asymptotic expansions, and \cite{ZKO1981,ZKO1982} for the underlying convergence and averaging theory. Qualitative gradient corrector results for time-dependent $H$-converging operators are given in \cite[Theorem~3.8]{DallAglioMurat1997}. Quantitative expansions with smoothing and dual correctors for space-time periodic coefficients are developed in \cite[Sections~2 and~3]{GengShen2017}; a parabolic construction of flux correctors is given in \cite[Section~2.4]{Xu2023}.

For elliptic equations on bounded {domains satisfying \eqref{eqn:CDC}}, \cite[Sections~4 and~5]{HaraEllipticHomogenization} combines boundary decay with an interior corrector estimate to obtain uniform convergence under $H$-convergence and a periodic $L^\infty$ convergence rate. Here we combine the parabolic corrector constructions with the {estimates under \eqref{eqn:CDC}} established above. The local $L^2$ condition on the forcing term yields an interior Morrey bound for the residual, which gives global H\"older convergence and a corrected-gradient Morrey rate. The scalar bounds for the chosen potentials and the residual estimate are proved below.

\subsection{Setting and main convergence results}\label{periodic-note:sec:setting}
In this section, $\|\cdot\|_{L^\infty}$ denotes the norm on $D\times(0,T)$ unless a local neighborhood is explicitly specified. Cell and cylinder norms refer to the displayed sets.
Under the standing assumption \eqref{eqn:finite-global-scale}, assume \eqref{eqn:CDC} and let $A(y,\tau)$ be a real measurable $\mathbb Z^{n+1}$-periodic matrix satisfying \eqref{eqn:A}. Set
\[
 A^\epsilon(x,t)=A(x/\epsilon,t/\epsilon^2),\qquad
 \mathcal H_\epsilon=\partial_t-\divergence(A^\epsilon\nabla),\qquad
 \mathcal H_0=\partial_t-\divergence(A_{\mathrm{hom}}\nabla),
\]
where the effective matrix is defined below. {For any periodic scalar, vector, or matrix field $h$, we write}
{
\[
 h^\epsilon(x,t):=h(x/\epsilon,t/\epsilon^2).
\]
In particular, spatial derivatives of the rescaled correctors satisfy
\[
 \nabla_x\omega_j^\epsilon
 =\epsilon^{-1}(\nabla_y\omega_j)^\epsilon.
\]
}
The geometry and the cylinder $Q(X,r)$ are as in \eqref{eqn:geometry}.
Fix $0<\beta<1$ and suppose
\begin{equation}\label{eqn:periodic-data}
 \|f\|_{L^2(Q(X,r))}\le M\delta(X)^{-1}r^{n/2+\beta},
 \qquad X\in D\times(0,T),\quad 0<r<\delta(X)/4.
\end{equation}
The correctors $\omega_j$ and the effective matrix $A_{\mathrm{hom}}$ are defined in Subsection~\ref{periodic-note:sec:cell}. Let $u_\epsilon$ and $u_{\mathrm{hom}}$ solve $\mathcal H_\epsilon u_\epsilon=f$ and $\mathcal H_0u_{\mathrm{hom}}=f$, respectively, with zero initial and lateral values attained continuously. Below we write $u=u_{\mathrm{hom}}$.

We apply {Theorem~\ref{thm:main}, Theorem~\ref{thm:divergence}, and Lemma~\ref{lem:std-comparison}}, uniformly to $A^\epsilon$ and $A_{\mathrm{hom}}$. Let $\alpha_{\mathrm{bd}}>0$ be the structural boundary exponent. Choose
\[
 0<\alpha_{\mathrm{ss}}<\min\{\beta,\alpha_{\mathrm{bd}},1\},\qquad
 \alpha_{\mathrm{gh}}=\min\{\alpha_{\mathrm{in}},\alpha_{\mathrm{ss}}\},
\]
where $\alpha_{\mathrm{in}}>0$ is a common admissible scalar interior exponent. In particular $0<\alpha_{\mathrm{gh}}\le\alpha_{\mathrm{ss}}<\beta<1$.
Theorem~\ref{thm:main} gives
\begin{equation}\label{eqn:periodic-global}
 [u_\epsilon]_{C^{\alpha_{\mathrm{gh}},\alpha_{\mathrm{gh}}/2}(\overline{D\times(0,T)})}
 +[u]_{C^{\alpha_{\mathrm{gh}},\alpha_{\mathrm{gh}}/2}(\overline{D\times(0,T)})}
 \le CM\|\delta\|_{L^\infty(D\times(0,T))}^{\beta-\alpha_{\mathrm{gh}}}.
\end{equation}
Since both solutions vanish on the parabolic boundary, it also gives
\begin{equation}\label{eqn:periodic-boundary}
 |u_\epsilon(X)|+|u(X)|\le CM\|\delta\|_{L^\infty(D\times(0,T))}^{\beta-\alpha_{\mathrm{gh}}}\delta(X)^{\alpha_{\mathrm{gh}}}.
\end{equation}
Indeed, the Cauchy--Schwarz inequality in \eqref{eqn:periodic-data} gives
$\int_{Q(X,r)}|f|\,dY\le CMr^{n+\beta}$.
For divergence forcing $\operatorname{div}{\bm{F}}$, Theorem~\ref{thm:divergence} gives
\begin{equation}\label{periodic-note:divinput}
 |v(X)|\le CM'\delta(X)^{\alpha_{\mathrm{ss}}}
\end{equation}
for the bounded continuous solution with zero initial and lateral values whenever
$\|{\bm{F}}\|_{L^2(Q(X,r))}\le M' r^{n/2+\alpha_{\mathrm{ss}}}$ on admissible cylinders.  

\begin{theorem}\label{thm:periodic-main}
Assume \eqref{eqn:finite-global-scale}, \eqref{eqn:CDC}, \eqref{eqn:A}, and \eqref{eqn:periodic-data}, with the periodic coefficients and zero initial and lateral values specified in this subsection. For every $0<\alpha<\agh$ and $0<\epsilon\le\|\delta\|_{L^\infty(D\times(0,T))}$,
\begin{equation}\label{eqn:periodic-holder-rate}
 [u_\epsilon-u_{\mathrm{hom}}]_{C^{\alpha,\alpha/2}(\overline{D\times(0,T)})}
 \le CM\|\delta\|_{L^\infty(D\times(0,T))}^{\beta-\agh}
          \epsilon^{\agh-\alpha}.
\end{equation}
Moreover,
\begin{equation}\label{eqn:periodic-rate}
 \|u_\epsilon-u_{\mathrm{hom}}\|_{L^\infty(D\times(0,T))}
 \le CM\|\delta\|_{L^\infty(D\times(0,T))}^{\beta-\agh}\epsilon^\agh.
\end{equation}
The constants depend only on {$n$, $L$, $\gamma$, $\beta$, and the chosen exponents}; the constant in \eqref{eqn:periodic-holder-rate} also depends on $\alpha$.
\end{theorem}

Fix $0<\alpha<\alpha_{\mathrm{gh}}$. Use $\mathsf V^{2,n+2\alpha}(D\times(0,T))$ from Subsection~\ref{subsec:mixed-source-space}, with $\beta$ replaced by $\alpha$ and with the same normalization.
Choose a continuous cutoff $\chi_\epsilon$, valued in $[0,1]$, equal to zero when $\delta\le C_3\epsilon$ and to one when $\delta\ge2C_3\epsilon$. The fixed structural constant $C_3$ is large enough for the local convolutions used below. Set
\begin{equation}\label{periodic-note:gradienterror}
 \bm{E}_\epsilon=\nabla u_\epsilon-\nabla u_{\mathrm{hom}}
 -\chi_\epsilon(\nabla_y\omega_j)^\epsilon\partial_j u_{\mathrm{hom}}.
\end{equation}
\begin{theorem}\label{periodic-note:gradienttheorem}
Under the assumptions of Theorem~\ref{thm:periodic-main},
\begin{equation}\label{periodic-note:gradientrate}
\begin{aligned}
 \|\bm{E}_\epsilon\|_{\mathsf V^{2,n+2\alpha}(D\times(0,T))}
 \le{}&CM\|\delta\|_{L^\infty(D\times(0,T))}^{\beta}\\
 &\times\left(\frac{\epsilon}{\|\delta\|_{L^\infty(D\times(0,T))}}\right)^{\alpha_{\mathrm{gh}}-\alpha}.
\end{aligned}
\end{equation}
The constant also depends on the chosen $\alpha$ and the gaps from the exponents in the interior estimates.
\end{theorem}

For the proof normalize $\|\delta\|_{L^\infty(D\times(0,T))}=1$. Under dilation by $\|\delta\|_{L^\infty(D\times(0,T))}$, the datum constant becomes $M\|\delta\|_{L^\infty(D\times(0,T))}^\beta$ and the period becomes $\epsilon/\|\delta\|_{L^\infty(D\times(0,T))}$. We restore this scaling at the end.

The normalization applies throughout Subsections~\ref{periodic-note:sec:local}--\ref{periodic-note:sec:gradient}. The supplementary compact-support energy statement below uses its own unscaled domain.
\subsection{Cell fields and their regularity}\label{periodic-note:sec:cell}
The cell problem and effective matrix follow the standard periodic construction in \cite[Chapter~2]{BLP1978} and \cite[Section~2]{GengShen2017}. Use the periodic cell $(0,1)^{n+1}$. Repeated spatial indices run from $1$ to $n$ and are summed. Define the mean-zero periodic cell solutions by
\begin{equation}\label{periodic-note:cell}
 \partial_\tau\omega_j-\operatorname{div}_y(A(e_j+\nabla_y\omega_j))=0,
 \qquad \int_{(0,1)^{n+1}}\omega_j\,dy\,d\tau=0.
\end{equation}
Put
\[
 (A_{\mathrm{hom}})_{ij}=\int_{(0,1)^{n+1}}[A(e_j+\nabla_y\omega_j)]_i\,dy\,d\tau,
 \quad B_{ij}=[A(e_j+\nabla_y\omega_j)]_i-(A_{\mathrm{hom}})_{ij}.
\]
Thus $\partial_iB_{ij}=\partial_\tau\omega_j$ and $\int_{(0,1)^{n+1}}B_{ij}\,dy\,d\tau=0$.
We use the parabolic potential construction of \cite[Section~2.4]{Xu2023}, with the spatial and temporal components written separately. For the mean-zero periodic solutions of
\[
 (\partial_\tau-\Delta_y)P_{ij}=B_{ij},\qquad
 (\partial_\tau-\Delta_y)P_j=\omega_j,
\]
define
\begin{equation}\label{periodic-note:definition}
 \Phi_{kij}=\partial_iP_{kj}-\partial_kP_{ij},\qquad
 \Psi_{ij}=P_{ij}-\partial_iP_j.
\end{equation}
The divergence of the first auxiliary equation gives $\partial_iP_{ij}=\partial_\tau P_j$. Consequently,
\begin{equation}\label{periodic-note:identities}
\begin{aligned}
 B_{ij}&=\partial_k\Phi_{kij}+\partial_\tau\Psi_{ij},
 &\partial_i\Psi_{ij}&=\omega_j,\\
 \Phi_{kij}&=-\Phi_{ikj},
 &\Phi_{kij}&=\partial_i\Psi_{kj}-\partial_k\Psi_{ij}.
\end{aligned}
\end{equation}
In particular $(\partial_\tau-\Delta_y)\Psi_{ij}=B_{ij}-\partial_i\omega_j$. These equations, including the curl relation in \eqref{periodic-note:identities}, fix the constructed potentials modulo constants. Their space-time means are zero by \eqref{periodic-note:definition}. The tensor $\Psi$ is not antisymmetric in its spatial indices.

We choose $\Phi$ and $\Psi$ together through \eqref{periodic-note:definition}, so that the flux identities and the curl relation in \eqref{periodic-note:identities} hold for the same potentials. Let $v$ denote the spatially smoothed macroscopic solution defined in Subsubsection~\ref{periodic-note:subsec:expansion}. The temporal divergence identity makes the second correction $\epsilon^2\Psi_{ij}^\epsilon\partial_{ij}v$ absorb a time derivative, while spatial antisymmetry cancels the third derivatives contracted against $\Phi$. The curl relation in \eqref{periodic-note:identities} expresses $\Phi$ in terms of $\Psi$ and ensures that the spatial and temporal flux decompositions use the same potentials. For this choice, scalar cell regularity yields the additional bounds in Lemma~\ref{periodic-note:potentialbounds}, which control both the residual and the gradient of the second correction.

Use the exponent $\alpha_{\mathrm{in}}$ chosen in Subsection~\ref{periodic-note:sec:setting}, small enough also for the scalar correctors. Applied to $y_j+\omega_j$, scalar regularity and the Caccioppoli inequality \cite{LSU} give the cell bounds
\begin{equation}\label{periodic-note:cellgrowth}
 \int_{Q(Z,r)}(|\nabla_y\omega|^2+|B|^2)\,dy\,d\tau
 \le Cr^{n+2\alpha_{\mathrm{in}}},
 \qquad \|\omega\|_{L^\infty((0,1)^{n+1})}\le C.
\end{equation}
The growth estimate holds for $0<r\le1$. Here $Q(Z,r)$ is a backward parabolic cylinder in the periodic extension. The constants are uniform in $Z$.

\begin{lemma}\label{periodic-note:potentialbounds}
The potentials in \eqref{periodic-note:definition} satisfy
\[
 \|\Phi\|_{L^\infty((0,1)^{n+1})}+\|\Psi\|_{L^\infty((0,1)^{n+1})}
 +\|\nabla_y\Psi\|_{L^\infty((0,1)^{n+1})}\le C.
\]
\end{lemma}
\begin{proof}
The Cauchy--Schwarz inequality and \eqref{periodic-note:cellgrowth} give
$\int_{Q(Z,r)}|B|\,dy\,d\tau\le Cr^{n+1+\alpha_{\mathrm{in}}}$.
The spatial gradient of the heat kernel has size $Cr^{-n-1}$ on a parabolic annulus of radius $r$. Thus the small-scale contributions to $\nabla_yP_{ij}$ are bounded by a convergent sum $C\sum_{m\ge0}2^{-m\alpha_{\mathrm{in}}}$. The same argument bounds $P_{ij}$. On the spatial mean-zero modes the periodic heat semigroup decays exponentially at large times. The spatially constant mode of $P_{ij}$ is the periodic primitive of the spatial mean of $B_{ij}$, and is bounded by its $L^1$ norm over a period. This proves boundedness of $P_{ij}$ and $\nabla_yP_{ij}$.

Scalar interior regularity gives a parabolic H\"older bound for $\omega_j$. Constant-coefficient parabolic Schauder estimates for $(\partial_\tau-\Delta_y)P_j=\omega_j$ therefore bound $\nabla_yP_j$ and $D_y^2P_j$. Definition \eqref{periodic-note:definition} proves the assertion. 
\end{proof}

\subsection{Common local estimates}\label{periodic-note:sec:local}
\subsubsection{Local regularity estimates for the homogenized solution}
Fix $Z\in D\times(0,T)$ and put $\rho=\delta(Z)\le1$. Use $Q(Z,\rho/8)$ as an outer cylinder and $Q(Z,\rho/32)$ as an inner cylinder. On the outer cylinder, $\delta$ is comparable to $\rho$ and the forcing constant in \eqref{eqn:periodic-data} is at most $CM\rho^{-1}$. For every $Y\in Q(Z,\rho/32)$ and $0<r<\rho/128$, the constant-coefficient equation gives
\begin{align}
 \|D_x^2u\|_{L^2(Q(Y,r))}+\|\partial_tu\|_{L^2(Q(Y,r))}
 &\le CM\rho^{\alpha_{\mathrm{gh}}-1-\beta}r^{n/2+\beta},\label{periodic-note:localhess}\\
 \|\nabla u\|_{L^\infty(Q(Z,\rho/32))}
 &\le CM\rho^{\alpha_{\mathrm{gh}}-1},\qquad
 [\nabla u]_{C^{\beta,\beta/2}(Q(Z,\rho/32))}
 \le CM\rho^{\alpha_{\mathrm{gh}}-1-\beta}.\label{periodic-note:localgrad}
\end{align}
The constants are independent of {$Z$, $\rho$, and $r$}. These local norms are taken on the displayed cylinders, rather than on the full domain. Spatial shifts of size at most $\epsilon<\rho/1024$ from the inner cylinder remain in the outer cylinder; the same estimates hold on such shifted neighborhoods after changing only the fixed numerical ratios.

For the heat potential $w$ of a datum $g$ for $\mathcal H_0$ on $\R^{n+1}$, the standard $W^{2,1}_2$ estimate gives
\[
 \|D_x^2w\|_{L^2(\R^{n+1})}+\|\partial_tw\|_{L^2(\R^{n+1})}
 \le C\|g\|_{L^2(\R^{n+1})};
\]
see \cite[Chapter~IV, Section~3]{LSU} and \cite[Chapter~VII]{Lieberman1996}. In the constant-coefficient $L^2$ case this also follows directly from the Plancherel identity and ellipticity. We use this estimate to obtain the local growth bounds as follows.

Restrict $f$ to $Q(Z,\rho/8)$ and extend this restriction by zero to $\R^{n+1}$. This localization requires no time extension of $f$ beyond the upper face. For radii smaller than a fixed multiple of $\rho$, recenter cylinders meeting this support and use \eqref{eqn:periodic-data}. For larger radii, use the total $L^2$ bound of the localized datum. A finite covering at scale $\rho$ therefore gives the whole-space Morrey bound $CM\rho^{-1}r^{n/2+\beta}$. Construct its heat potential for $\mathcal H_0$. To estimate the second spatial derivatives and time derivative on $Q(Y,r)$, split the datum into its part on $Q(Y,2r)$ and the complementary part. Apply the $L^2$ estimate to the first part and the differentiated heat-kernel bounds to the second part. On the annulus of radius $2^mr$, the $L^1$ mass is at most $CM\rho^{-1}(2^mr)^{n+1+\beta}$, while a second spatial or first temporal derivative of the kernel is at most $C(2^mr)^{-n-2}$. Its $L^2(Q(Y,r))$ contribution is consequently at most $CM\rho^{-1}r^{n/2+\beta}2^{m(\beta-1)}$. Thus the far contributions sum because $\beta<1$. 

For the gradient, let $d>0$ be the parabolic distance between two points in $Q(Z,\rho/32)$. Split the localized source into a neighborhood of radius comparable to $d$ about these points and its complement. On the near part, the gradient kernel bound and the $L^1$ growth of the source give a contribution at most
\[
 CM\rho^{-1}\sum_{m\ge0}(2^{-m}d)^\beta
 \le CM\rho^{-1}d^\beta.
\]
On the far part, a gradient kernel difference at parabolic distance $d$ is bounded by $Cd a^{-n-2}$ at scale $a\ge d$. The spatial difference uses a second spatial derivative of the heat kernel; the temporal difference uses its gradient time derivative and $|t-s|\le d^2$. The resulting annular contributions satisfy
\[
 CM\rho^{-1}d\sum_{m\ge0}(2^md)^{\beta-1}
 \le CM\rho^{-1}d^\beta.
\]
The near sum converges because $\beta>0$ and the far sum because $\beta<1$. Thus the gradient of the localized potential has parabolic $C^{\beta,\beta/2}$ seminorm at most $CM\rho^{-1}$ on the inner cylinder. The heat kernels and their derivative estimates are given in \cite[Chapter~IV, Sections~1 and~2]{LSU}.

Subtract this potential from $u$. The difference is caloric on a smaller cylinder. Its oscillation is bounded by $CM\rho^{\alpha_{\mathrm{gh}}}$ by \eqref{eqn:periodic-boundary} and the potential bound $CM\rho^\beta$. Caloric derivative estimates give pointwise second spatial and first temporal derivatives bounded by $CM\rho^{\alpha_{\mathrm{gh}}-2}$. For $0<r<\rho/128$, their $L^2$ contributions satisfy
\[
 \rho^{\alpha_{\mathrm{gh}}-2}r^{(n+2)/2}
 \le \rho^{\alpha_{\mathrm{gh}}-1-\beta}r^{n/2+\beta}.
\]
The gradient and its H\"older seminorm follow from the same caloric estimates and the gradient-potential estimates. This proves \eqref{periodic-note:localhess}--\eqref{periodic-note:localgrad}. All estimates use past times only, and extend to an upper time face by limits from below.

\subsubsection{A common local expansion}\label{periodic-note:subsec:expansion}
We adapt the smoothed two-scale expansion in \cite[Theorem~2.2 and equation~(2.14)]{GengShen2017}, using the parabolic potentials in \eqref{periodic-note:definition}. Here the convolution is spatial and local. Applying it also to $u$ introduces a forcing difference, which we retain in divergence form below.
Choose a symmetric $\eta\in C_c^\infty(B(0,1))$ with $\eta\ge0$ and $\int_{\R^n}\eta(x)\,dx=1$. Set $\eta_\epsilon(x)=\epsilon^{-n}\eta(x/\epsilon)$ and $v=\eta_\epsilon*_xu$ wherever $\delta>C_0\epsilon$. The identities below hold in distributions on that interior region; they also hold for smooth compactly supported $u$ on a whole-time cylinder.
For this local convolution, define
\[
 w_\epsilon=u_\epsilon-v-\epsilon\omega_j^\epsilon\partial_jv
                -\epsilon^2\Psi_{ij}^\epsilon\partial_{ij}v.
\]
The forcing difference has the divergence representation
\[
 f-\eta_\epsilon*_xf=\operatorname{div}\bm{F}_\epsilon,\qquad
 \bm{F}_\epsilon(x,t)=\int_{\R^n}\eta_\epsilon(y)y
           \int_0^1 f(x-sy,t)\,ds\,dy.
\]
Define the residual vector field $\bm{r}_\epsilon$ by its components:
\begin{equation}\label{periodic-note:residualfield}
\begin{aligned}
 r_{\epsilon,k}={}&F_{\epsilon,k}
  +\epsilon A_{k\ell}^\epsilon\omega_j^\epsilon\partial_{\ell j}v
  +\epsilon\Phi_{kij}^\epsilon\partial_{ij}v\\
 &-\epsilon^2\Psi_{kj}^\epsilon\partial_j\partial_tv
  +A_{k\ell}^\epsilon\partial_\ell
        (\epsilon^2\Psi_{ij}^\epsilon\partial_{ij}v).
\end{aligned}
\end{equation}
The cell identities imply
\begin{equation}\label{periodic-note:resid}
 \mathcal H_\epsilon w_\epsilon=\operatorname{div}\bm{r}_\epsilon.
\end{equation}
To verify \eqref{periodic-note:resid}, the terms left after the first correction are
\begin{align*}
 B_{ij}^\epsilon\partial_{ij}v
       -\epsilon\omega_j^\epsilon\partial_j\partial_tv
 &=\epsilon\partial_k(\Phi_{kij}^\epsilon\partial_{ij}v)
       +\partial_t(\epsilon^2\Psi_{ij}^\epsilon\partial_{ij}v)
       -\partial_k(\epsilon^2\Psi_{kj}^\epsilon\partial_j\partial_tv).
\end{align*}
The spatial third derivatives in the first term cancel by antisymmetry in {$k$ and $i$}. Subtracting $\epsilon^2\Psi_{ij}^\epsilon\partial_{ij}v$ cancels its time derivative and produces the last flux in \eqref{periodic-note:resid}.

\subsubsection{Morrey bounds for the residual and corrections}
For $\delta(Y)\asymp\rho\ge C_1\epsilon$ and $r<c\rho$,
\begin{equation}\label{periodic-note:fluxmorrey}
 \|\bm{r}_\epsilon\|_{L^2(Q(Y,r))}
 \le CM\epsilon\rho^{\alpha_{\mathrm{gh}}-1-\beta}r^{n/2+\beta}.
\end{equation}
To prove it, apply Minkowski's integral inequality to every spatial convolution. Centers shifted by at most $\epsilon$ have boundary distance comparable to $\rho$ and their cylinders stay in the neighborhoods where \eqref{periodic-note:localhess} holds. Thus convolution preserves its local Morrey bound, including when $r<\epsilon$. Moving one derivative to the kernel costs $C\epsilon^{-1}$. The two terms $\epsilon A_{k\ell}^\epsilon\omega_j^\epsilon\partial_{\ell j}v$ and $\epsilon\Phi_{kij}^\epsilon\partial_{ij}v$ use boundedness of {$A$, $\omega$, and $\Phi$} and \eqref{periodic-note:localhess}. In $\epsilon^2\Psi_{kj}^\epsilon\partial_j\partial_tv$, move $\partial_j$ to the convolution kernel and use the estimate for $\partial_tu$ in \eqref{periodic-note:localhess}. Expanding the last term in \eqref{periodic-note:residualfield} gives
\[
 \epsilon A_{k\ell}^\epsilon(\partial_{y_\ell}\Psi_{ij})^\epsilon\partial_{ij}v
 +\epsilon^2 A_{k\ell}^\epsilon\Psi_{ij}^\epsilon\partial_{\ell ij}v.
\]
The first term uses boundedness of $\nabla_y\Psi$; for the second, move $\partial_\ell$ to the kernel. Both are controlled by $C\epsilon$ times the local $L^2$ growth bound for $D_x^2u$. The field $\bm{F}_\epsilon$ is bounded by $C\epsilon$ times the local Morrey norm of $f$. Its prefactor $M\rho^{-1}$ is at most $M\rho^{\alpha_{\mathrm{gh}}-1-\beta}$, since $\rho\le1$ and $\alpha_{\mathrm{gh}}<\beta$. Lemma~\ref{periodic-note:potentialbounds} now proves \eqref{periodic-note:fluxmorrey}.

The correction is pointwise small:
\begin{equation}\label{periodic-note:correctionsmall}
 |v-u|+|\epsilon\omega_j^\epsilon\partial_jv|+|\epsilon^2\Psi_{ij}^\epsilon\partial_{ij}v|
 \le CM\epsilon\delta^{\alpha_{\mathrm{gh}}-1}.
\end{equation}
The first two terms follow from \eqref{periodic-note:localgrad}. For the last term subtract a constant gradient inside the differentiated convolution to obtain
\[
 |D_x^2v|\le C\epsilon^{\beta-1}[\nabla u]_{C_x^\beta},\qquad
 |\epsilon^2\Psi_{ij}^\epsilon\partial_{ij}v|\le CM\epsilon^{1+\beta}\rho^{\alpha_{\mathrm{gh}}-1-\beta}
       \le CM\epsilon\rho^{\alpha_{\mathrm{gh}}-1}.
\]
The potentials are continuous by their heat-potential construction, and $D_x^2v$ is continuous by the spatial H\"older bound and the continuity of $\nabla u$. Hence the corrected error has continuous values on the artificial parabolic boundaries used below. The residual identity and the local energy regularity justify weak comparison.

\subsubsection{A multiplier estimate from the cell equation}
\begin{lemma}\label{periodic-note:multiplier}
For $d\in C_c^\infty(\R^{n+1})$ and every $j$, with the $L^2$ norms in this lemma taken over $\R^{n+1}$,
\begin{equation}\label{periodic-note:mult}
 \|\,| (\nabla_y\omega_j)^\epsilon|\,|\nabla d|\,\|_{L^2}
 \le C\bigl(\|\nabla d\|_{L^2}+\epsilon\|D_x^2d\|_{L^2}
                         +\epsilon\|\partial_td\|_{L^2}\bigr).
\end{equation}
\end{lemma}
\begin{proof}
For a fixed $j$, put
$K_1=\epsilon^2\int_{\R^{n+1}}|\nabla\omega_j^\epsilon|^2|\nabla d|^2\,dX$.
The rescaled cell equation is
\[
 \epsilon^2\partial_t(\omega_j^\epsilon)-\epsilon^2\operatorname{div}(A^\epsilon\nabla\omega_j^\epsilon)
 =\epsilon\operatorname{div}(A^\epsilon e_j).
\]
Test this equation with $(\omega_j^\epsilon)|\nabla d|^2$. The time term satisfies
\begin{align*}
 \epsilon^2\int_{\R^{n+1}}\partial_t(\omega_j^\epsilon)\,(\omega_j^\epsilon)|\nabla d|^2\,dX
 &=-\epsilon^2\int_{\R^{n+1}}(\omega_j^\epsilon)^2\nabla d\cdot\nabla\partial_td\,dX\\
 &=\epsilon^2\int_{\R^{n+1}}
 (2(\omega_j^\epsilon)\nabla\omega_j^\epsilon\cdot\nabla d+(\omega_j^\epsilon)^2\Delta d)\partial_td\,dX.
\end{align*}
For weak cell solutions the first line denotes the weighted energy pairing; time averages justify the identity. Spatial integration by parts is valid since $(\omega_j^\epsilon)$ is bounded, $\nabla\omega_j^\epsilon$ is locally square integrable, and $d$ is smooth with compact support.

Set {$a=\|\nabla d\|_{L^2}$, $b=\|D_x^2d\|_{L^2}$, and $c=\|\partial_td\|_{L^2}$}.
{Boundedness of $(\omega_j^\epsilon)$, boundedness of $A$, and ellipticity} give
\[
 K_1\le C K_1^{1/2}(a+\epsilon b+\epsilon c)
       +C\epsilon ab+C\epsilon^2bc.
\]
The diffusion cross term is bounded by $C\epsilon K_1^{1/2}b$, the forcing terms by $CK_1^{1/2}a+C\epsilon ab$, and the time term by $C\epsilon K_1^{1/2}c+C\epsilon^2bc$. Young's inequality now yields \eqref{periodic-note:mult}.
\end{proof}

\subsubsection{Local estimates for removing the convolution}
Fix $0<\alpha<\alpha_{\mathrm{gh}}$ as in Theorem~\ref{periodic-note:gradienttheorem}. Let $\rho=\delta(Y)\ge C\epsilon$ and write $d=v-u$. On neighborhoods of boundary distance comparable to $\rho$, the coefficient in \eqref{periodic-note:localhess} is $CM\rho^{\alpha_{\mathrm{gh}}-1-\beta}$. For $\epsilon\le r<c\rho$, symmetry of $\eta$ and Minkowski's integral inequality give
\begin{align*}
 \|d\|_{L^2(Q(Y,r))}
 &\le CM\epsilon^2\rho^{\alpha_{\mathrm{gh}}-1-\beta}r^{n/2+\beta},\\
 \|\nabla d\|_{L^2(Q(Y,r))}
 &\le CM\epsilon\rho^{\alpha_{\mathrm{gh}}-1-\beta}r^{n/2+\beta},\\
 \|D_x^2d\|_{L^2(Q(Y,r))}+\|\partial_td\|_{L^2(Q(Y,r))}
 &\le CM\rho^{\alpha_{\mathrm{gh}}-1-\beta}r^{n/2+\beta}.
\end{align*}
All these bounds allow a fixed enlargement of the cylinder. Localizing Lemma~\ref{periodic-note:multiplier} at scale $r$ therefore gives
\begin{equation}\label{periodic-note:largemult}
 \|(\nabla_y\omega_j)^\epsilon\partial_jd\|_{L^2(Q(Y,r))}
 \le CM\epsilon\rho^{\alpha_{\mathrm{gh}}-1-\beta}r^{n/2+\beta},
\end{equation}
for $\epsilon\le r<c\rho$. To justify this localization, multiply $d$ by a spatial and temporal cutoff supported on an enlarged cylinder and apply Lemma~\ref{periodic-note:multiplier}. The additional terms involve {$d/r$, $\epsilon\nabla d/r$, and $\epsilon d/r^2$}, and are controlled by the three convolution estimates at the start of this subsubsection. If a time extension is needed, reflection across the upper time face on the enlarged cylinder preserves these local $W^{2,1}_2$ bounds. Choose the cutoff to equal one on the smaller cylinder. There its product with $d$ has gradient exactly $\nabla d$.

For $r<\epsilon$, the cell growth bound \eqref{periodic-note:cellgrowth} and the gradient H\"older estimate give instead
\begin{align*}
 \|(\nabla_y\omega_j)^\epsilon\|_{L^2(Q(Y,r))}
 &\le C\epsilon^{1-\alpha_{\mathrm{in}}}r^{n/2+\alpha_{\mathrm{in}}},\\
 |\nabla d|&\le CM\epsilon^\beta\rho^{\alpha_{\mathrm{gh}}-1-\beta}.
\end{align*}
Their product is at most
$CM\epsilon^{\alpha_{\mathrm{gh}}-\alpha}r^{n/2+\alpha}$, because $r<\epsilon\le C\rho$ and $\alpha<\alpha_{\mathrm{in}}$. For $r\ge\epsilon$, the same target bound follows from \eqref{periodic-note:largemult}, using $r<c\rho$ and $\rho\ge C\epsilon$.

The remaining terms in removing the convolution satisfy
\begin{equation}\label{periodic-note:remaininggradientterms}
\begin{aligned}
 &\|\nabla d\|_{L^2(Q(Y,r))}
 +\|\epsilon\omega_j^\epsilon\nabla\partial_jd\|_{L^2(Q(Y,r))}\\
 &\quad+\|\nabla(\epsilon^2\Psi_{ij}^\epsilon\partial_{ij}v)\|_{L^2(Q(Y,r))}
 +\|\epsilon\omega_j^\epsilon\nabla\partial_ju\|_{L^2(Q(Y,r))}\\
 &\hspace{12mm}\le CM\epsilon\rho^{\alpha_{\mathrm{gh}}-1-\beta}r^{n/2+\beta}
 \le CM\epsilon^{\alpha_{\mathrm{gh}}-\alpha}r^{n/2+\alpha}.
\end{aligned}
\end{equation}
For $\nabla d$ the first inequality holds also when $r<\epsilon$: write the difference of gradients as an integral of $D_x^2u$ along spatial translations and use \eqref{periodic-note:localhess}. The term $\epsilon\omega_j^\epsilon\nabla\partial_jd$ uses the bound for $D_x^2d$ at the start of this subsubsection. For the second correction, expand its gradient into $\epsilon(\nabla_y\Psi_{ij})^\epsilon\partial_{ij}v$ and $\epsilon^2\Psi_{ij}^\epsilon\nabla\partial_{ij}v$; apply Lemma~\ref{periodic-note:potentialbounds} and \eqref{periodic-note:localhess}, moving one derivative to the kernel in the latter term. The final term uses boundedness of $\omega$ and \eqref{periodic-note:localhess} directly.

\subsection{Proof of the global H\"older convergence theorem}\label{periodic-note:sec:supremum}
\subsubsection{Divergence data in a boundary layer}
Suppose ${\bm{F}}$ is supported in $E_\rho=\{\rho\le\delta\le2\rho\}$ and
\[
 \|{\bm{F}}\|_{L^2(Q(Y,r))}\le M'\rho^{1-\beta}r^{n/2+\beta},
 \qquad Y\in E_\rho,\quad r<c\rho.
\]
Then its solution with zero original parabolic boundary values satisfies
\begin{equation}\label{periodic-note:layerbound}
 \|v\|_{L^\infty}\le CM'\rho.
\end{equation}
Let $Q(X,r)$ be admissible and meet the source layer at $Y$. Since $\delta$ is $1$-Lipschitz for $d_p$,
\[
 3r<\delta(X)-r\le\delta(Y)\le2\rho,
\]
so $r<2\rho/3$. For $r$ smaller than a sufficiently small fixed multiple of $\rho$, recenter at source points and enlarge the radius by a fixed factor. To respect the backward time orientation, choose source points in the cylinder whose time coordinates increase to the supremum of the source times. A backward cylinder of radius $3r$ about each such point contains the part of the original source cylinder below that time. Apply the assumed bound and pass to the limit in the integral of $|{\bm{F}}|^2$. For the remaining radii, which are comparable to $\rho$, cover the spatial ball and time interval by a fixed number of pieces of spatial diameter $c\rho/8$ and time length $(c\rho/8)^2$. The same recentering on each piece gives a bound at radius less than $c\rho$. The number of pieces depends only on the fixed ratios and dimension. Thus
\[
 \|{\bm{F}}\|_{L^2(Q(X,r))}
 \le CM'\rho^{1-\beta}r^{n/2+\beta}
 \le CM'\rho^{1-\alpha_{\mathrm{ss}}}r^{n/2+\alpha_{\mathrm{ss}}}.
\]
The distributional support of $\divergence{\bm{F}}$ is contained in the relative closure of $E_\rho$, so $\|\delta\|_{L^\infty(\operatorname{supp}(\operatorname{div}{\bm{F}}))}\le2\rho$. Apply Theorem~\ref{thm:divergence} with growth constant $CM'\rho^{1-\alpha_{\mathrm{ss}}}$. Its support bound gives \eqref{periodic-note:layerbound} directly.

\subsubsection{Summation of layers and comparison}
Choose $R\ge C_2\epsilon$, where $C_2$ ensures all the preceding convolution neighborhoods are interior. {For each nonnegative integer $m$, set}
\[
 \rho_m=2^mR,\quad E_m=\{\rho_m<\delta\le2\rho_m\},\quad
 \bm{r}_\epsilon^{(m)}=\bm{1}_{E_m}\bm{r}_\epsilon.
\]
Only finitely many layers are nonempty because $\delta\le1$. By \eqref{periodic-note:fluxmorrey} and \eqref{periodic-note:layerbound}, their solutions with zero initial and Dirichlet boundary values satisfy
\[
 \|v_\epsilon^{(m)}\|_{L^\infty}\le CM\epsilon\rho_m^{\alpha_{\mathrm{gh}}-1}.
\]
Therefore $v_{\epsilon,R}=\sum_mv_\epsilon^{(m)}$ satisfies
\begin{equation}\label{periodic-note:layersum}
 \|v_{\epsilon,R}\|_{L^\infty}\le CM\epsilon R^{\alpha_{\mathrm{gh}}-1},
\end{equation}
since $\alpha_{\mathrm{gh}}<1$. On $\{\delta>R\}=D_R\times(R^2,T)$, where
$D_R=\{x:\operatorname{dist}(x,\partial D)>R\}$, the function $w_\epsilon-v_{\epsilon,R}$ is homogeneous. On its parabolic boundary, \eqref{eqn:periodic-boundary} and \eqref{periodic-note:correctionsmall} give
\[
 |w_\epsilon|\le CM(R^{\alpha_{\mathrm{gh}}}+\epsilon R^{\alpha_{\mathrm{gh}}-1}).
\]
The comparison principle and \eqref{periodic-note:layersum} imply the same bound for $w_\epsilon$ throughout this interior region, with an additional factor of two on $\|v_{\epsilon,R}\|_{L^\infty}$. Recover $u_\epsilon-u$ by \eqref{periodic-note:correctionsmall}. In the remaining region $\delta\le R$, use \eqref{eqn:periodic-boundary} directly. Hence
\[
 \|u_\epsilon-u\|_{L^\infty}\le CM(R^{\alpha_{\mathrm{gh}}}+\epsilon R^{\alpha_{\mathrm{gh}}-1}).
\]
If $T=\infty$, comparison is applied on finite time strips and the supremum is then taken over them. Lemma~\ref{lem:std-comparison} applies on each such strip.

For $\epsilon<C_2^{-1}$ take $R=C_2\epsilon$. For larger $\epsilon\le1$, \eqref{eqn:periodic-boundary} already gives the same rate after increasing the constant. Undoing the normalization yields \eqref{eqn:periodic-rate}. Interpolate this estimate with the common $C^{\agh,\agh/2}$ bound \eqref{eqn:periodic-global}, using the parabolic metric. This gives \eqref{eqn:periodic-holder-rate} and completes the proof of Theorem~\ref{thm:periodic-main}.

\begin{remark}
The supremum rate also gives the scalar conclusion
\begin{equation}\label{periodic-note:solutionmorreyrate}
\begin{aligned}
 &\bigl\||u_\epsilon-u_{\mathrm{hom}}|^2\,dX
       \bigr\|_{\mathsf M^{n+2\beta}(D\times(0,T))}^{1/2}\\
 &\qquad\le CM\|\delta\|_{L^\infty(D\times(0,T))}^{1+\beta-\alpha_{\mathrm{gh}}}
                    \epsilon^{\alpha_{\mathrm{gh}}}.
\end{aligned}
\end{equation}
Indeed, the local $L^2$ norm of the error is at most its supremum norm times $Cr^{(n+2)/2}$, and $r^{1-\beta}\le C\|\delta\|_{L^\infty(D\times(0,T))}^{1-\beta}$. Multiplication by the normalization factor $\|\delta\|_{L^\infty(D\times(0,T))}^{\beta}$ proves \eqref{periodic-note:solutionmorreyrate}.\end{remark}

\subsection{Proof of the corrected-gradient Morrey rate}\label{periodic-note:sec:gradient}
We prove Theorem~\ref{periodic-note:gradienttheorem}, using the common expansion and local removal estimates from Subsection~\ref{periodic-note:sec:local}.
\subsubsection{Interior energy decay for the corrected error}
Fix an interior cylinder $Q(Y,\rho/8)$ and suppose that $\mathcal H_\epsilon w=\operatorname{div}\bm F$ there. For $0<R\le\rho/16$, let $h$ solve the homogeneous equation in $Q(Y,R)$ with the same parabolic boundary values as $w$. Testing the equation for $w-h$, which has zero parabolic boundary values, gives
\[
 \int_{Q(Y,R)}|\nabla(w-h)|^2\,dX
 \le C\int_{Q(Y,R)}|\bm F|^2\,dX.
\]
For $0<r<R/4$, the homogeneous oscillation estimate and the Caccioppoli inequality, together with the parabolic Poincar\'e inequality for $h$, give
\[
 \int_{Q(Y,r)}|\nabla h|^2\,dX
 \le C(r/R)^{n+2\alpha_{\mathrm{in}}}
             \int_{Q(Y,R)}|\nabla h|^2\,dX.
\]
Indeed, subtract a constant on the outer cylinder, use the homogeneous $L^2$ oscillation estimate on an intermediate cylinder, and apply the Caccioppoli inequality on $Q(Y,2r)$. The factors are $r^{-2}r^{n+2}(r/R)^{2\alpha_{\mathrm{in}}}$ and $R^{-n}\int_{Q(Y,R)}|\nabla h|^2\,dX$. Combining the last two estimates gives
\[
 \int_{Q(Y,r)}|\nabla w|^2\,dX
 \le C(r/R)^{n+2\alpha_{\mathrm{in}}}
             \int_{Q(Y,R)}|\nabla w|^2\,dX
       +C\int_{Q(Y,R)}|\bm F|^2\,dX.
\]
Suppose that $\|\bm F\|_{L^2(Q(Y,a))}\le M'a^{n/2+\beta}$ for $0<a\le\rho/8$. Fix $0<\alpha<\min\{\alpha_{\mathrm{in}},\beta\}$ and choose $0<\theta<1/4$ so that
\[
 C\theta^{n+2\alpha_{\mathrm{in}}}
 \le \tfrac12\theta^{n+2\alpha}.
\]
Apply the preceding energy decay at the successive radii {$R$, $\theta R$, $\theta^2R$, and $\ldots$}. After division by the radius to the power $n+2\alpha$, the homogeneous contraction factor is at most $1/2$ and the forcing terms contain the additional factors $\theta^{2j(\beta-\alpha)}$. Summing the resulting geometric series and using monotonicity at intermediate radii gives, for $0<r<R/2$ and $R\le\rho/16$,
\[
\begin{aligned}
 \int_{Q(Y,r)}|\nabla w|^2\,dX
 \le{}&C(r/R)^{n+2\alpha}
              \int_{Q(Y,R)}|\nabla w|^2\,dX\\
 &+C(M')^2R^{2(\beta-\alpha)}r^{n+2\alpha}.
\end{aligned}
\]
At the outer scale, the Caccioppoli inequality between $Q(Y,R)$ and $Q(Y,2R)$ yields
\[
 \int_{Q(Y,R)}|\nabla w|^2\,dX
 \le CR^n\|w\|_{L^\infty(Q(Y,2R))}^2
       +C(M')^2R^{n+2\beta}.
\]
Take $R=\rho/16$. Combining the last two bounds and taking square roots therefore gives, for $0<r<\rho/64$,
\begin{equation}\label{periodic-note:localenergyiteration}
 \|\nabla w\|_{L^2(Q(Y,r))}
 \le C\left(\rho^{-\alpha}\|w\|_{L^\infty(Q(Y,\rho/8))}
               +M'\rho^{\beta-\alpha}\right)r^{n/2+\alpha}.
\end{equation}
A constant may be subtracted from $w$ in this supremum norm. If the residual estimate is initially available only for $a<c\rho$ with a smaller fixed $c$, replace {$\rho/8$, $\rho/16$, and $\rho/64$} throughout this argument by fixed smaller multiples of $\rho$. The constant then also depends on those fixed ratios.

For $w_\epsilon$, take $c>0$ small enough that all residual estimates hold on $Q(Y,2c\rho)$. The global rate and \eqref{periodic-note:correctionsmall} imply
\[
 \|w_\epsilon\|_{L^\infty(Q(Y,c\rho))}
 \le CM\epsilon^{\alpha_{\mathrm{gh}}},\qquad
 M'=CM\epsilon\rho^{\alpha_{\mathrm{gh}}-1-\beta},
 \quad \rho=\delta(Y)\ge C\epsilon.
\]
Use \eqref{periodic-note:fluxmorrey} in \eqref{periodic-note:localenergyiteration}. Since $\alpha<\alpha_{\mathrm{gh}}\le\alpha_{\mathrm{in}}$ and $\alpha<\beta$, this gives
\begin{equation}\label{periodic-note:correctedmorrey}
\begin{aligned}
 \|\nabla w_\epsilon\|_{L^2(Q(Y,r))}
 &\le CM\left(\epsilon^{\alpha_{\mathrm{gh}}}\rho^{-\alpha}
       +\epsilon\rho^{\alpha_{\mathrm{gh}}-1-\alpha}\right)r^{n/2+\alpha}\\
 &\le CM\epsilon^{\alpha_{\mathrm{gh}}-\alpha}r^{n/2+\alpha}.
\end{aligned}
\end{equation}

\subsubsection{Recovery of the unsmoothed gradient}
On the region where $\chi_\epsilon=1$, differentiation of the definition of $w_\epsilon$ gives
\begin{align*}
 \bm{E}_\epsilon={}&\nabla w_\epsilon+\nabla d
 +(\nabla_y\omega_j)^\epsilon\partial_jd
 +\epsilon\omega_j^\epsilon\nabla\partial_jd\\
 &+\nabla(\epsilon^2\Psi_{ij}^\epsilon\partial_{ij}v)
 +\epsilon\omega_j^\epsilon\nabla\partial_ju.
\end{align*}
Estimates \eqref{periodic-note:correctedmorrey}, \eqref{periodic-note:largemult}, and \eqref{periodic-note:remaininggradientterms} prove the desired local bound there. A finite covering deals with admissible radii up to $\delta(Y)/4$.

\subsubsection{The boundary and transition layers}
By \eqref{eqn:periodic-data} and the Cauchy--Schwarz inequality, the measure $|f|\,dX$ satisfies \eqref{eqn:morrey} with growth constant $CM$. Thus Lemma~\ref{lem:mixed-gradient}, with $\nu=f\,dX$ and $\bm F=0$, applies to both original solutions. The exponent is the common $\alpha_{\mathrm{gh}}$ fixed in Subsection~\ref{periodic-note:sec:setting}; under the normalization $\|\delta\|_{L^\infty}=1$, it gives
\begin{equation}\label{periodic-note:rawgradientgrowth}
 \|\nabla u_\epsilon\|_{L^2(Q(X,r))}
 +\|\nabla u\|_{L^2(Q(X,r))}
 \le CM r^{n/2+\alpha_{\mathrm{gh}}}.
\end{equation}
The estimate holds for every $X\in D\times(0,T)$ and $0<r<\delta(X)/4$, with constants uniform in $\epsilon$.
For cylinders whose centers have $\delta(X)\le C\epsilon$, their admissible radii satisfy $r\le C'\epsilon$. Hence \eqref{periodic-note:rawgradientgrowth} is at most
$CM\epsilon^{\alpha_{\mathrm{gh}}-\alpha}r^{n/2+\alpha}$.

Where the cutoff term is nonzero in such cylinders, $\delta\ge C_3\epsilon$, so \eqref{periodic-note:localgrad} gives $|\nabla u|\le CM\epsilon^{\alpha_{\mathrm{gh}}-1}$. For $r<\epsilon$, multiply this bound by the rescaling of \eqref{periodic-note:cellgrowth}. For $\epsilon\le r\le C'\epsilon$, use the periodic cell energy bound instead. Both yield
\[
 \|\chi_\epsilon(\nabla_y\omega_j)^\epsilon\partial_ju\|_{L^2(Q(X,r))}
 \le CM\epsilon^{\alpha_{\mathrm{gh}}-\alpha}r^{n/2+\alpha}.
\]
This covers the boundary layer and the transition of the cutoff, including cylinders meeting the transition. Together with the interior estimates it proves
\[
 \sup_{\substack{X\in D\times(0,T)\\0<r<\delta(X)/4}}
 r^{-n/2-\alpha}\|\bm{E}_\epsilon\|_{L^2(Q(X,r))}
 \le CM\epsilon^{\alpha_{\mathrm{gh}}-\alpha}.
\]
Undoing the parabolic dilation and using the norm normalization in Subsection~\ref{subsec:mixed-source-space} proves \eqref{periodic-note:gradientrate}. For periods comparable to the global distance scale, the uniform bounds in this argument suffice after increasing the constant.

\subsection{The compact-support energy estimate}\label{periodic-note:sec:appendix}
We conclude with an energy estimate for a smooth homogenized solution whose support stays inside the spatial domain. Unlike Theorems~\ref{thm:periodic-main} and~\ref{periodic-note:gradienttheorem}, this result uses a whole-time energy class and requires no boundary layer estimate. Let $U\subset\R^n$ be bounded and open, and work on $U\times\R$.
In this subsection, $\|\cdot\|_{L^2}$ denotes the norm on $U\times\R$. 
Solutions with zero spatial boundary values are understood in
$L^2(\R;H_0^1(U))$, with time derivative in $L^2(\R;H^{-1}(U))$.
The finite-energy condition uniquely determines the solution; equivalently, it is the causal solution for the compactly supported forcing considered below.

\begin{theorem}\label{periodic-note:main}
Let $u=u_{\mathrm{hom}}\in C_c^\infty(U\times\R)$, set $f=\mathcal H_0u$, and let $\mathcal H_\epsilon u_\epsilon=f$ with zero spatial boundary values. Let $K_x$ be the spatial projection of $\operatorname{supp}u$ and suppose
$0<\epsilon<\min\{1,\operatorname{dist}(K_x,\partial U)/2\}$.
Then
\begin{align}
 \|\nabla(u_\epsilon-u-\epsilon\omega_j^\epsilon\partial_ju)\|_{L^2}
 &\le C\epsilon\bigl(\|D_x^2u\|_{L^2}+\|\partial_tu\|_{L^2}\bigr),\label{periodic-note:finalgrad}\\
 \|u_\epsilon-u\|_{L^2}
 &\le C_U\epsilon\bigl(\|D_x^2u\|_{L^2}+\|\partial_tu\|_{L^2}\bigr).\label{periodic-note:finall2}
\end{align}
Here $C$ depends only on {dimension, ellipticity, and the resulting scalar cell bounds}; $C_U$ also depends on a Poincar\'e constant for $U$. The restriction on $\epsilon$ ensures that the spatial convolution remains supported inside $U$.
\end{theorem}

\subsubsection{Energy estimate for the two-scale error}
Extend $u$ by zero outside $U$ and use the symmetric spatial kernel fixed in Subsubsection~\ref{periodic-note:subsec:expansion}. Put $v=\eta_\epsilon*_xu$ and
\[
 K=\|D_x^2u\|_{L^2}+\|\partial_tu\|_{L^2}.
\]
The support restriction in Theorem~\ref{periodic-note:main} ensures that $v$ has compact support in $U\times\R$. Use the corrected error and residual defined in Subsubsection~\ref{periodic-note:subsec:expansion}.
Lemma~\ref{periodic-note:potentialbounds} and convolution estimates give
\begin{equation}\label{periodic-note:fluxbounds}
\begin{aligned}
 \|\bm{r}_\epsilon\|_{L^2}
 &\le C\epsilon\|f\|_{L^2}
      +C\epsilon\|D_x^2v\|_{L^2}
      +C\epsilon^2\|\nabla\partial_tv\|_{L^2}
      +C\epsilon^2\|D_x^3v\|_{L^2}\\
 &\le C\epsilon K.
\end{aligned}
\end{equation}
We also record the estimate used to recover the error after subtracting the first corrector:
\begin{equation}\label{periodic-note:secondcorrectionenergy}
 \|\nabla(\epsilon^2\Psi_{ij}^\epsilon\partial_{ij}v)\|_{L^2}
 \le C\epsilon\|D_x^2v\|_{L^2}+C\epsilon^2\|D_x^3v\|_{L^2}
 \le C\epsilon\|D_x^2u\|_{L^2}.
\end{equation}
The residual identity \eqref{periodic-note:resid} holds in distributions. The function $w_\epsilon$ has zero spatial trace and belongs to $L^2(\R;H_0^1(U))$; \eqref{periodic-note:resid} then puts its time derivative in $L^2(\R;H^{-1}(U))$.
Testing \eqref{periodic-note:resid} with $w_\epsilon$ over the whole time axis gives
\[
 \int_{U\times\R} A^\epsilon\nabla w_\epsilon\cdot\nabla w_\epsilon\,dX
 =-\int_{U\times\R}\bm{r}_\epsilon\cdot\nabla w_\epsilon\,dX.
\]
The temporal energy term vanishes by the whole-time energy identity, justified by time cutoffs and the finite-energy class. Consequently,
\begin{equation}\label{periodic-note:Wbound}
 \|\nabla w_\epsilon\|_{L^2}\le C\epsilon K.
\end{equation}

\subsubsection{Conclusion of the energy estimate}
Put $d=v-u$. The elementary convolution estimates give
\[
 \|\nabla d\|_{L^2}\le C\epsilon\|D_x^2u\|_{L^2},
 \qquad \|D_x^2d\|_{L^2}+\|\partial_td\|_{L^2}\le 2K.
\]
By Lemma~\ref{periodic-note:multiplier} and boundedness of $\omega$,
\[
 \|\nabla(\epsilon\omega_j^\epsilon\partial_jd)\|_{L^2}
 \le C\epsilon K.
\]
Using the exact identity
\[
 u_\epsilon-u-\epsilon\omega_j^\epsilon\partial_ju
 =w_\epsilon+\epsilon^2\Psi_{ij}^\epsilon\partial_{ij}v+d+\epsilon\omega_j^\epsilon\partial_jd,
\]
we obtain \eqref{periodic-note:finalgrad} from \eqref{periodic-note:secondcorrectionenergy} and \eqref{periodic-note:Wbound}.
All terms have zero spatial trace. The Poincar\'e inequality therefore gives \eqref{periodic-note:finall2}, using also
$\|\nabla u\|_{L^2}\le C_U\|D_x^2u\|_{L^2}$, since each $\partial_ju$ has compact support in $U$.

The argument extends by approximation to the $W^{2,1}_2$ closure of smooth functions supported in a fixed compact subset of $U\times\R$. Lemma~\ref{periodic-note:multiplier} controls the first corrector under this approximation.

\endgroup

\section*{Acknowledgments}
This work was supported by JSPS KAKENHI Grant Numbers 21K03281 and 23H03798.
The author used ChatGPT (OpenAI) to assist with mathematical discussions,
manuscript preparation, and English-language editing.
\enlargethispage{4\baselineskip}
\bibliographystyle{amsplain}
\bibliography{references}
\end{document}